\documentclass[runningheads]{llncs}

\usepackage[latin1]{inputenc}
\usepackage[T1]{fontenc}
\usepackage[english]{babel}
\usepackage{enumitem}   

\usepackage{mathrsfs}
\usepackage{amscd}
\usepackage{amsfonts}
\usepackage{amsmath}
\usepackage{amssymb}
\usepackage{amstext}
\usepackage{amsbsy}

\usepackage{xspace}
\usepackage[all]{xy}
\usepackage{graphicx}
\usepackage{tikz}
\usepackage{url}
\usepackage{latexsym}

\usepackage{graphicx}

\newtheorem{fact}{fact}
\newtheorem{thm}[fact]{Theorem}
\newtheorem{prop}[fact]{Proposition}
\newtheorem{defini}[fact]{Definition}

\begin{document}
\title{Lower Bounds on $\beta(\alpha)$ and other properties of $\alpha$-register machines}
%\title{Lower bounds on $\beta(\alpha)$ and other properties of $\alpha$-ITRMs}
\author{Merlin Carl\inst{1}}
\authorrunning{M. Carl}
\institute{Institut f\"ur mathematische, naturwissenschaftliche und technische Bildung, Abteilung f\"ur 
Mathematik und ihre Didaktik, Europa-Universit\"at Flensburg, Germany 
\email{merlin.carl@uni-flensburg.de}}
\maketitle              % typeset the header of the contribution
\begin{abstract} 
This paper extends our paper \cite{C2} for the conference ``Computability in Europe'' 2022. 

After Infinite Time Turing Machines (ITTM) were introduced in Hamkins and Lewis \cite{HL}, a number of machine models of computability have been generalized to the transfinite, along with various variants thereof. While for some of these models the computational strength has been successfully determined, there are still several white spots on the map of transfinite computability. In this paper, we %give an overview of the known results and 
contribute to the understanding of the computational strength of transfinite machine models by (i) proving lower bounds on the computational strength of $\alpha$-Infinite Time Register Machines ($\alpha$-ITRMs) for certain values of $\alpha$, refuting a conjecture about their strength made in \cite{alpha itrms}, 
%(ii) showing that the computational strength of weak OTMs is equal to that of OTMs 
(ii) showing that the computational strength of cardinal-recognizing ITRMs is equal to that of ITRMs and 
(iii) showing that non-solvability of the bounded halting problem, existence of a universal machine and an increase of computational power by allowing machines to recognize cardinals are equivalent for $\alpha$-ITRMs for all relevant values of $\alpha$ .

%%For an ordinal $\alpha$, an $\alpha$-ITRM is a machine model of transfinite computability that operates on finitely many registers, each of which can contain an ordinal $\rho<\alpha$; they were introduced by Koepke in \cite{K1}. 
%More specficially, in \cite{alpha itrms}, it was shown that the $\alpha$-ITRM-computable subsets of $\alpha$ are exactly those in a level $L_{\beta(\alpha)}$ of the constructible hierarchy. It was conjectured in \cite{alpha itrms} that $\beta(\alpha)$ is the first limit of admissible ordinals above $\alpha$. 
%Here, we show that this is (dramatically) false; in particular, even the computational strength of $\omega^{\omega}$-ITRMs goes far beyond $\omega_{\omega}^{\text{CK}}$. To this end, we prove lower bounds on this computational strength, using a strategy for iterating $\alpha$-ITRM-computable operators for $\eta$ many steps on $\alpha^{\eta}$-ITRMs. We show how these bounds can in certain cases be lifted to the uncountable, using techniques from Abramson and Sacks \cite{AS}. These results were contained in \cite{C2}; the other results in this paper are new. 

Finally, we given some results indicating how the picture changes when the use of parameters is dropped or restricted. 

%A remarkable feature of $\omega$-ITRMs is that the halting problem for machines with a fixed number of registers is solvable by a machine with some additional registers (see Koepke and Miller \cite{KM}). This rules out the existence of a universal $\omega$-ITRM (\cite{KM}). We show that, in a certain weak sense, the existence of a universal machine and the solvability of the bounded halting problem form a dichotomy for $\alpha$-ITRMs for all sufficiently closed ordinals $\alpha$. 

\end{abstract}
%
%
%

%PLAN: 
%schwache $\alpha$-ITTMs mit und ohne Parameter (Verallgemeinerung auf $\alpha$ ist neu, mit und ohne Parameter)
%$\alpha$-ITTMs mit und ohne Parameter [ohne Parameter ist irgendwie neu] 
%starke und schwache $\alpha$-ITRMs mit und ohne Parameter [neu ist der Iterationstrick, ebenso die Betrachtung ohne Parameter]

%ERGAENZEN: BH-DICHOTOMIE UND SIMULATION WIEDER REIN. MEHR GOSTANIAN-BEGRUENDUNG. EVTL. ANFANG AUSBAUEN UND MEHR ZU MASCHINENTYPEN UND IHRER BERECHNUNGSSTAERKE SCHREIBEN.

\section{Introduction}

Ordinal computability studies generalizations of models of computability to the transfinite, thereby connecting set theory (in particular descriptive set theory and constructibility) and computability theory. The models of computability studied in ordinal computability include 
Infinite Time Turing Machines (ITTMs) (Hamkins and Lewis, \cite{HL}), Ordinal Turing Machines (OTMs) (Koepke, \cite{K}), $\alpha$-Turing Machines (Koepke and Seyfferth, \cite{KS}), $\alpha$-ITTMs (Koepke, \cite{K1}), $(\alpha,\beta)$-ITTMs (\cite{KS}, \cite{COW}), 
weak and strong Infinite Time Register Machines (Koepke \cite{K}, Koepke and Miller \cite{KM}), Ordinal Register Machines (Koepke and Syders \cite{KSy}), $\alpha$-(w)ITRMs (Koepke, \cite{K1}), $(\alpha,\beta)$-(w)ITRMs (ibid.), Infinite Time Blum-Shub-Smale-Machines (Koepke and Seyfferth, \cite{KS1}), Surreal Blum-Shub-Smale-Machines (Galeotti and Nobrega, \cite{GN}), Ordinal $\lambda$-Calculus (Fischbach and Seyfferth, \cite{FS}) and Deterministic Ordinal Automata (\cite{C5}). Given this great diversity of models, one might be excused to utter the objection that, in contrast to classical computability theory with Turing computability as its central notion, this is a zoo rather than a model, and the area lacks coherence, thus significantly reducing the interest of results about such models. 

%TO DO [from CiE 2020 Talk]: Ordinal computability studies generalizations of models of computability to the transfinite. Examples of Models: Infinite Time Turing Machines, Ordinal Turing Machines, $\alpha$-Turing Machines, $\alpha$-ITTMs, $(\alpha,\beta)$-ITTMs, 
%(weak and strong Infinite Time Register Machines, Ordinal Register Machines, $\alpha$-(w)ITRMs, $(\alpha,\beta)$-(w)ITRMs, Infinite Time Blum-Shub-Smale-Machines, Surreal Blum-Shub-Smale-Machines, Ordinal $\lambda$-Calculus, Deterministic Ordinal Automata, ... 

%Objection: `There is not one model, but a zoo. The area lacks coherence.'

Counter to this view, we offer the following perspective: Among the ``maximal'' models, with fully ordinalized ressources, tape models, register models, $\lambda$-calculus etc. all lead to the same notion of computability, which coincides with constructibility (see, e.g., Fischbach \cite{F}). The other models should be viewed as resource-bounded versions of this one, stable notion of computability, arising, e.g., by restricting the available time or the available space.\footnote{But also in other ways, for example, by stipulating that the content of a tape cell may only change finitely often.} They are then not analogous to Turing machines, but to complexity classes, which indeed form a ``zoo'' in the classical setting, albeit one very worthy of study. The main difference to the finitary case is that, in the transfinite, constant resource bounds, such as restricting the tape length to $\omega$, lead to interesting and stabel notions. But this is a feature, rather than a bug, of ordinal computability. 

Another point of critique is that ordinal computability is focused too much on its different models rather than general topics of computability: one should rather look at concepts, rather than machines. To this, we reply that the growing literature on ordinal computability contains works, for example, on transfinite versions of degree theory (\cite{HL1}, \cite{W2}), algorithmic randomness (\cite{CS}, \cite{AM}), computable model theory \cite{HMSW}, Weihrauch reducibility \cite{C3}, \cite{GN}, complexity theory (\cite{DHS}, \cite{C4}) and realizability (\cite{CGP}); and that it has fruitfully interacted with such areas as constructibility theory \cite{K2}, descriptive set theory (\cite{CH}, \cite{CSW2}) and, recently, proof theory (\cite{CGP}, \cite{P}). Thus, there is no shortage of conceptual work, applications and interactions.\footnote{See also \cite{C}, chapter $8$.} 

Still, the ``zoo'' of models, which has now been around for about 15 years, leaves us with several challenging open problems. The computational strength of Turing machines with tape length $\alpha$ was considered in \cite{R}, \cite{CRS} and, with time additionally restricted to $\beta$, determined in \cite{COW}. One benefit of such research is that it often leads (i) to new characterizations of known types of ordinals and classes (for example, by Koepke \cite{K}, the hyperarithmetic real numbers are exactly those computable by a wITRM; recently, the ordinal $\gamma^{1}_{2}$ defined by Kechris, Marker and Sami was characterized as the supremum of the countable halting time bounds of ITTMs that semidecide some set of real numbers, \cite{CSW1}, \cite{CSW2}) and (ii) to new classes and types of ordinals, such as the ordinals $\lambda$, $\zeta$ and $\Sigma$ introduced by Welch in his analysis of ITTMs or the class of ITTM-decidable sets of real numbers. However, as we will see in the next section, for the register models, large spots on the map are still white.

The aim of this paper, then, is to contribute to the classification of models of transfinite computability by their computational strength. %So far, this led to new characterizations of well-known types of ordinals, including large countable ordinals (see, e.g. Madore \cite{M}).

%View: At the ``top'', with fully ordinalized ressources, tape models, register models, $\lambda$-calculus etc. do indeed agree with respect to computational power.

%ITTMs, (w)ITRMs, $\alpha$-(w)ITRMs, $\alpha$-ITTMs, $\alpha$-TMs etc. should be viewed as ressource-bounded versions of this notion of ordinal computability. They are then not analogous to Turing machines, but to complexity classes. 
%In this sense, classical computability is also concerned with a ``zoo'': P, NP, PSPACE, LOGSPACE, EXPTIME, ...

%The difference to the finitary case is that, in the transfinite, constant ressource bounds (``tape length $\omega$'', ``$\alpha$ many steps'') lead to interesting notions. But that is not a bug, it is a feature!

%Another point of critique: fixation on properties of models rather than general topics of computability (look at concepts, not at machines). 

%Ordinal computability has defined and studied degree theory, randomness, computable model theory, Weihrauch reducibility, realizability;
% it also has strong relations to constructibility theory, forcing, large cardinals. 

%Still, the following will be a old-school consideration of specific details of models. This is part of a general project to ``clean up'' the area by finishing the systematic classification of models started $20$ years ago. 

The section \ref{iteration section} on iterating $\alpha$-computable operators, along with a part of the introduction (in particular, the next one) and some of the open questions, are taken from the CiE 2022 conference paper \cite{C2}. The rest of the material, unless indicated otherwise, is an original contribution of this paper. 

\subsection{Register Models of Transfinite Computability}

%TO DO: WAS WAR SCHON IM CIE-PAPER, WAS IST NEU?

%Mention: Corollary 48 of \cite{alphaitrms} was that, for an index $\alpha$, we have $\beta(\alpha)\geq\alpha^{+\omega}$; this will be greatly improved.

In \cite{K1}, Koepke introduced resetting $\alpha$-Infinite Time Register Machines, abbreviated $\alpha$-ITRMs; an extensive discussion can be found in \cite{C}. Such machines have finitely many registers, each of which can store a single ordinal smaller than $\alpha$. Programs for $\alpha$-ITRMs are just programs for classical register machines 
as introduced, e.g., in Cutland \cite{Cu} and consist of finitely many enumerated program lines, each of which contains one of the following commands: (i) an incrementation operation, which increases the content of some register by $1$, (ii) a copy instruction, which replaces the content of one register 
by that of another, (iii) a conditional jump, which changes the active program line to a certain value when the contents of two finite sequences of registers\footnote{That we allow the simultaneous comparison of two finite sequences, rather than just two, registers, has again technical advantages explained in \cite{alpha itrms}, p. 2} are equal and otherwise proceeds with the next program line, (iv) an oracle command, which checks whether the content of some register is contained in the oracle and changes 
the content of that register to $1$ if that is the case and otherwise to $0$.\footnote{Note that the ``reset'' command for replacing the content of a register by $0$ can be carried out by having a register with value $0$ and using the copy instruction; for this reason, it is not included here, in contrast to the account in \cite{K1}.} 
For technical reasons, we start the enumeration of the program lines with $1$ rather than $0$. 

The computation of an $\alpha$-ITRM then works as follows: At successor stages, we simply carry out the program as we would in a classical (finite) register machine.\footnote{If $\alpha$ is a successor ordinal, the incrementation operation may lead to the register content $\alpha$; in that case, the content is replaced by $0$. However, we will be mostly concerned with limit values of $\alpha$ in this paper.} At limit stages, the content of each register is the inferior limit of the sequence of earlier contents of this register; if this happens to be $\alpha$, we say that the register ``overflows'' and set its content to $0$. The active program line is just the inferior limit of the sequence of earlier active program lines. In the case $\alpha=\omega$, one drops the prefix and merely speaks of ITRMs, which have been studied in detail (\cite{CFKMNW}, \cite{K1}, \cite{KM}).

 There is also a weaker model for register computations on $\alpha$, known as ``weak'' or ``unresetting'' $\alpha$-ITRMs, abbreviated $\alpha$-wITRMs. These differ from $\alpha$-ITRMs in that a register overflow -- i.e., if, due to a limit operation, or (if $\alpha$ is a successor ordinal) due to an incrementation step -- has the consequence that the computation is not defined. 
In the special case that $\alpha=\omega$, these are called wITRMs, which were introduced and studied in Koepke \cite{K}. The more general type was first mentioned in Koepke \cite{K1}, but received little attention until \cite{alpha itrms}. 
%for which the computation is undefined in this case. H

In \cite{K1}, Koepke showed that, for $\alpha=\omega$, the subsets of $\alpha$ computable by such an $\alpha$-ITRM are exactly those in $L_{\omega_{\omega}^{\text{CK}}}$. Further information on $\omega$-ITRMs was obtained in \cite{CFKMNW} and \cite{KM}. It is also known from Koepke and Siders \cite{ORM} that, when one lets $\alpha$ be On, i.e., when one imposes no restriction on the size 
of register contents, the computable sets of ordinals are exactly the constructible ones. Recently, strengthening a result in \cite{C}, it was shown in \cite{alpha itrms} that the $\alpha$-ITRM-computable subsets of $\alpha$ coincide with those in $L_{\alpha+1}$ if and only if $L_{\alpha}\models\text{ZF}^{-}$\footnote{I.e., ZF set theory without the power set axiom; for the subtleties of the axiomatization, see \cite{GJH}.}; and moreover, it was shown that, for any exponentially closed $\alpha$, 
the $\alpha$-ITRM-computable subsets of $\alpha$ are exactly those in $L_{\beta(\alpha)}$, where $\beta(\alpha)$ is the supremum of the $\alpha$-ITRM-halting times, which coincides with the supremum of the ordinals that have $\alpha$-ITRM-computable codes. To determine the computational strength of $\alpha$-ITRMs for some exponentially closed ordinal $\alpha$, one thus needs to determine $\beta(\alpha)$. 
However, except for the cases $\alpha=\omega$, $\alpha=\text{On}$ and $L_{\alpha}\models$ZF$^{-}$, no value of $\beta(\alpha)$ is currently known. A reasonable conjecture compatible with all results obtained in \cite{alpha itrms} was that $\beta(\alpha)=\alpha^{+\omega}$, the first limit of admissible ordinals greater than $\alpha$, 
unless $L_{\alpha}\models$ZF$^{-}$, which would be the most obvious analogue of Koepke's result on $\omega$-ITRMs. This, however, will be shown to be false below. As a result, there is currently not even a good conjecture about what the values of $\beta(\alpha)$ might be, making the problem even more difficult.

Concerning weak register machines, we know from Koepke \cite{K} that the subsets of $\omega$ computable by a wITRM are precisely the hyperarithmetic ones, i.e., those contained in $L_{\omega_{1}^{\text{CK}}}$. It was then shown in \cite{alpha itrms} that, when $\alpha$ is $\Pi_3$-reflecting, the $\alpha$-wITRM-computable subsets of $\alpha$ are precisely those that are $\Delta_{1}$ over $L_{\alpha}$. Ordinals with the latter property are called ``$u$-weak'', and it is known that $u$-weak ordinal need not be $\Pi_3$-reflecting (\cite{C2}, Theorem $53$) (although they need to be admissible (\cite{C2}, Theorem 52), but not every admissible ordinal is $u$-weak (\cite{C2}, Theorem 53)). A full characterization of $u$-weak ordinals, or, more generally, of the computational strength of $\alpha$-ITRMs for values of $\alpha$ that are neither $\omega$ nor $\Pi_3$-reflecting is still wanting.

%MORE HERE: ERGEBNISSE ZUSAMMENFASSEN

In this paper, we will obtain lower bounds on the computational strength of $\alpha$-ITRMs by showing how, when $\alpha$ is exponentially closed, $\alpha$-ITRMs can compute transfinite (in fact $\alpha\cdot\omega$ long) iterations of $\beta$-ITRM-computable operators for $\beta<\alpha$. As a consequence, 
we are able to show that the conjecture mentioned above fails dramatically: In fact, for the first exponentially closed ordinal $\varepsilon_{0}$ larger than $\omega$, we will already have $\beta(\varepsilon_{0})\geq\omega_{\varepsilon_{0}\cdot\omega}^{\text{CK}}$, while the next limit of admissible ordinals after $\varepsilon_{0}$ is of course still $\omega_{\omega}^{\text{CK}}$. This improves Corollary 48 of \cite{alpha itrms}, where it was shown that $\beta(\alpha)\geq\alpha^{+\omega}$ when $\alpha$ is an index ordinal.\footnote{Thus, while the goal of our project is to ``tame'' the ``zoo'' of machine models, this paper rather indicates that the ``zoo'' may in fact be a jungle. We hope that this will attract adventurers.} 

We also show that $\alpha$-ITRMs are either able to solve the halting problem for $\alpha$-ITRMs restricted to programs using a fixed number $k$ of registers for every $k\in\omega$ or allow for a universal $\alpha$-ITRM. For $\alpha=\omega$, the first alternative is known to hold by Koepke and Miller \cite{KM} and for $L_{\alpha}\models$ZF$^{-}$, this follows from the results in \cite{alpha itrms}, while for $\alpha=\text{On}$, the second alternative holds by \cite{ORM}. We currently do not know which alternative holds for any other values of $\alpha$. We then offer a third characterization of these $\alpha$ by showing that the restricted halting problem is solvable precisely for those $\alpha$ for which granting the machine the ability to notice when its working time reaches a cardinal does not change their computational power. Such cardinal-recognizing variants were first considered for ITTMs by Habic, see \cite{Ha}. 

In the considerations so far, $\alpha$-(w)ITRMs computations are always understood to be allowed the use of parameters, i.e., some registers besides the input register may initially contain ordinals other than $0$. Dropping or restricting parameters, the picture changes considerably and resembles the situation for parameter-free tape models, so called $\alpha$-ITTMs, as described in \cite{R} and \cite{CRS}: In particular, there are pairs of ordinals $\alpha,\beta$ such that $\alpha$-ITRMs and $\beta$-ITRMs are incomparable with respect to their computational strength.

For an ordinal $\alpha$, we will write $\alpha^{+}$ to denote the smallest admissible ordinal strictly larger than $\alpha$. Moreover, for $\alpha,\iota\in\text{On}$, we recursively define $\alpha^{+0}=\alpha$, $\alpha^{+(\iota+1)}=(\alpha^{+\iota})^{+}$ and $\alpha^{+\iota}=\text{sup}_{\xi<\iota}\alpha^{+\xi}$ when $\iota$ is a limit ordinal. We use $p(\alpha,\beta)$ for Cantor's ordinal pairing function. For $x\subseteq\alpha$, $P$ an $\alpha$-(w)ITRM-program, $\rho<\alpha$, we write $P^{x}\downarrow=\iota$ to indicate that the program $P$, when run in the oracle $x$ and with $\rho$ initially in its first register, halts with $\iota$ in its first register, and we write $P^{x}(\rho)\uparrow$ to indicate that the computation does not halt; when $x$ is empty, the superscript is omitted.

%Explain the machines. In particular, explain $\alpha$-ITRMs when $\alpha=\hat{\alpha}+1$ is a successor ordinal: When a register contains $\hat{\alpha}$ and is incremented, it is reset to $0$. 

%auskommentiert. CiE-version sollte sich auf iteration und konsequenzen davon konzentrieren.
%\section{Simulating $\omega\cdot 2$-ITRMs on ITRMs}

\subsection{A survey of the computational strength of transfinite register machine models}

In this section, we want to summarize the results known so far on the computational strength of %tape and 
register models of transfinite computability. 

\begin{defini}{\label{comp strength}}
Recall that an ordinal $\beta$ is $\alpha$-(w)ITRM-clockable if and only if there is an $\alpha$-(w)ITRM-program $P$ and a parameter $\rho<\alpha$ such that $P(\rho)$ halts in precisely $\beta$ many steps.

Let us denote by COMP$_{\alpha-\text{wITRM}}$ and COMP$_{\alpha-\text{ITRM}}$ the sets of subsets of $\alpha$ computable by $\alpha$-wITRMs and $\alpha$-ITRMs (both with parameters), respectively. 

Moreover, we recall from \cite{alpha itrms} that $\beta(\alpha)$ denotes the supremum of the $\alpha$-ITRM-clockable ordinals, while $\beta^{w}(\alpha)$ denotes the supremum of the $\alpha$-wITRM-clockable ordinals. 
\end{defini}

We use a standard way of encoding transitive $\in$-structures as subsets of ordinals: Given a transitive $\in$-structure $S$ an ordinal $\alpha$ and a bijection $f:\alpha\rightarrow S$, we can code $S$ by $\{p(\iota,\xi):\iota,\xi\in\alpha\wedge f(\iota)\in f(\xi)\}$. When $\beta$ is closed under the pairing function, this will yield a subset of $\beta$, in which case it is called a $\beta$-code for $S$. We can then say that $S$ is $\alpha$-ITRM-computable if and only if $S$ has an $\alpha$-ITRM-computable $\alpha$-code.

We recall that ZF$^{-}$ denotes Zermelo-Fraenkel set theory without the power set axiom; more precisely, we use the axiomatization given in Gitman et al. \cite{GJH}.

Since every $\alpha$-wITRM-computation is also an $\alpha$-ITRM-computation, it is clear that the computational strength of $\alpha$-wITRMs is no greater than that of $\alpha$-ITRMs; and we see from the above that, when $\alpha=\omega$, the computational strength of $\alpha$-ITRMs considerably exceeds that of $\alpha$-wITRMs. We note that this holds for unboundedly many ordinals, while it fails for unboundedly many others. For most ordinals, we currently do not know the answer. 

\begin{prop}{\label{weak equals strong}}
\begin{enumerate}
\item There are unboundedly many ordinals $\alpha$ such that 

COMP$_{\alpha-\text{wITRM}}\subsetneq$COMP$_{\alpha-\text{ITRM}}$. 
\item There are unboundedly many ordinals $\alpha$ such that COMP$_{\alpha-\text{wITRM}}=$COMP$_{\alpha-\text{ITRM}}$. In fact, for each ordinal $\alpha$ and each $\gamma\in[\alpha+1,\alpha\omega)$, we have COMP$_{\gamma-\text{ITRM}}=$COMP$_{\gamma-\text{wITRM}}$. 
\end{enumerate}
\end{prop}
\begin{proof}
\begin{enumerate}
\item Let $\alpha$ be such that $L_{\alpha}\models$ZF$^{-}$. (It is easy to see that there are unboundedly many ordinals with this property; for instance, every regular cardinal is of this kind.) 
Now, by Theorem $19$ of \cite{alpha itrms}, we have COMP$_{\alpha-\text{ITRM}}=\mathfrak{P}(\alpha)\cap L_{\alpha+1}$. On the other hand, each such ordinal is clearly $\Pi_{3}$-reflecting, and thus, by Theorem $37$ of \cite{alpha itrms}, we have COMP$_{\alpha-\text{wITRM}}=\Delta_{1}(L_{\alpha})\cap\mathfrak{P}(\alpha)$. This is clearly a proper subset of $\mathfrak{P}(\alpha)\cap L_{\alpha+1}$. 

\item Let $\alpha=\beta+1$ be a successor ordinal. It is easy to see that a $\beta$-ITRM-computation can be simulated by an $\alpha$-wITRM-computation in which each register that contains $\beta$ is reset to $0$. Thus, we have COMP$_{\beta-\text{ITRM}}\subseteq$COMP$_{\alpha-\text{wITRM}}$. 

we will show in Proposition \ref{strong successor} below that COMP$_{\beta-\text{ITRM}}=$COMP$_{(\beta+1)-\text{ITRM}}$ for all ordinals $\beta$. From this, it follows that, for $\gamma\in[\alpha+2,\alpha\omega)$, we have:

COMP$_{\gamma-\text{wITRM}}\subseteq$COMP$_{\gamma-\text{ITRM}}=$\\COMP$_{(\alpha+1)-\text{ITRM}}\subseteq$COMP$_{(\alpha+2)-\text{wITRM}}\subseteq$COMP$_{\gamma-\text{wITRM}}$.

%However, we will show in Proposition \ref{strong successor} below that COMP$_{\beta-\text{ITRM}}=$COMP$_{(\beta+1)-\text{ITRM}}$ for all ordinals $\beta$, and so we get COMP$_{\alpha-\text{ITRM}}=$COMP$_{\beta-\text{ITRM}}\subseteq$COMP$_{\alpha-\text{wITRM}}\subseteq$COMP$_{\alpha-\text{ITRM}}$.
\end{enumerate}
\end{proof}

%Proposition \ref{strong successor} gives us some information on the question when $\alpha$-ITRMs are equivalent in computational power to $\alpha$-wITRMs:
%
%\begin{corollary}
%For each ordinal $\alpha$ and each $\gamma\in[\alpha+1,\alpha\omega)$, we have COMP$_{\gamma-\text{ITRM}}=$COMP$_{\gamma-\text{wITRM}}$. 
%\end{corollary}
%\begin{proof}
%For $\gamma\in[\alpha+2,\alpha\omega)$, we have, by the Proposition \ref{strong successor}: 
%
%COMP$_{\gamma-\text{wITRM}}\subseteq$COMP$_{\gamma-\text{ITRM}}=$COMP_{(\alpha+1)-\text{ITRM}}\subseteq$COMP$_{(\alpha+2)-\text{wITRM}}\subseteq$COMP$_{\gamma-\text{wITRM}}$.
%
%\end{proof}

\begin{question}
Characterize those ordinals $\alpha$ for which COMP$_{\alpha-\text{ITRM}}=$COMP$_{\alpha-\text{wITRM}}$.
\end{question}

The results known so far about the computational strength of $\alpha$-ITRMs and $\alpha$-wITRMs are the following:

\begin{defini}
An ordinal $\alpha$ is called (w)ITRM-singular if and only if there is an $\alpha$-(w)ITRM-computable cofinal function $f:\beta\rightarrow\alpha$ with $\beta<\alpha$. 
\end{defini}

\begin{thm}{\label{comp strength known}}

\begin{enumerate}[label=(\roman*)]
\item An ordinal $\alpha$ is ITRM-singular if and only if $L_{\alpha}\not\models\text{ZF}^{-}$ (\cite{alpha itrms}). 
\item If $\alpha$ is ITRM-singular, then an ordinal is $\alpha$-ITRM-clockable if and only if it is $\alpha$-ITRM-computable. 
\item COMP$_{\omega-\text{ITRM}}=\mathfrak{P}(\omega)\cap L_{\omega_{\omega}^{\text{CK}}}$, and $\beta(\omega)=\omega_{\omega}^{\text{CK}}$ (Koepke, \cite{K1}). 
\item COMP$_{\omega-\text{wITRM}}=\mathfrak{P}(\omega)\cap L_{\omega_{1}^{\text{CK}}}$, and $\beta^{w}(\omega)=\omega_{1}^{\text{CK}}$ (Koepke, \cite{K}). 
\item By slight abuse of notation, COMP$_{\text{On}-\text{ITRM}}=$COMP$_{\text{On}-\text{wITRM}}=\mathfrak{P}(\text{On})\cap L$. (Koepke, \cite{ORM})
\item COMP$_{\alpha-\text{ITRM}}=\mathfrak{P}(\alpha)\cap L_{\alpha+1}$ if and only if $L_{\alpha}\models\text{ZF}^{-}$ if and only if $\alpha$ is not ITRM-singular. In this case, we have $\beta(\alpha)=\alpha^{\omega}$ (\cite{alpha itrms}).
\item In all other cases, COMP$_{\alpha-\text{ITRM}}=\mathfrak{P}(\alpha)\cap L_{\beta(\alpha)}$ (\cite{alpha itrms}).
\item If $\alpha$ is $\Pi_{3}$-reflecting, then $\beta^{w}(\alpha)=\alpha$ and COMP$_{\alpha-\text{wITRM}}=\Delta_{1}(L_{\alpha})$ (\cite{alpha itrms}). 
\end{enumerate}
\end{thm}

For ($\omega$-)ITRMs, we have the following important result by Koepke and Miller \cite{KM}:

\begin{thm}{\label{bounded halting}}
[Koepke and Miller, \cite{KM}] For every $k\in\omega$, there is an ITRM-program that solves the halting problem for ITRM-programs using at most $k$ many registers. 
\end{thm}

\section{Lower bounds on jump ordinals for register models} 

It is easy to see that, if parameters are allowed, computations of $\alpha$-ITRMs can be simulated on $\beta$-wITRMs whenever $\beta>\alpha$, so that COMP$_{\alpha-\text{ITRM}}\subseteq$COMP$_{\beta-\text{wITRM}}$ for all ordinals $\alpha<\beta$. 

A natural task is then to determine those ordinals where the computational strength does actually increase, i.e., for each ordinal $\alpha$, the minimal ordinal $\alpha^{\prime}$ such that $\beta(\alpha^{\prime})>\beta(\alpha)$. In this section, we show that $\beta(\alpha)=\beta(\alpha^{\prime})$ whenever $\alpha^{\prime}\in[\alpha,\alpha\omega)$.

\begin{defini}
Let $\alpha$ be an ordinal. Then $\alpha^{j}_{\text{ITRM}}$, the ``ordinal ITRM-jump of $\alpha$'', denotes the minimal ordinal $\gamma$ such that $\beta(\gamma)>\beta(\alpha)$. 
\end{defini}

The ``jump''-terminology is justified by the following observation:\footnote{While it may be tempting to try to prove the next lemma using a simulation argument, simulating an arbitrary $\alpha$-ITRM-program on a $\gamma$-ITRM and clocking $\beta(\alpha)$ along the way. However, we point out that we know of no way to ``trade'' increased register capacity for extra register in a simulation; that is, we do not know how to simulate $k$ $\alpha$-registers with less than $k$ $\gamma$-registers, even if $\gamma$ is much larger than $\alpha$.} %koepke-siders: minimality considerations?}

\begin{lemma}
For each infinite $\alpha$, $\alpha^{j}_{\text{ITRM}}$ is the smallest ordinal $\gamma$ such that the halting problem for $\alpha$-ITRMs is solvable by a $\gamma$-ITRM. 
\end{lemma}
\begin{proof}
%However, we can simply 
We observe that, by Lemma $31$ from \cite{alpha itrms}, if $\beta(\gamma)>\beta(\alpha)$, there will be a $\gamma$-ITRM that clocks $\beta(\alpha)$, so $\beta(\alpha)$ is $\gamma$-ITRM-computable, and hence, 
so is a subset of $\alpha$ coding $L_{\beta(\alpha)}$ from which the halting set for $\alpha$-ITRMs is then $\gamma$-ITRM-computable.
\end{proof}

%\begin{remark}
%Note that the obvious proof idea of simulating an arbitrary $\alpha$-ITRM-program on a $\gamma$-ITRM and clocking $\beta(\alpha)$ along the way does not work, since we do not know how to simulate $\alpha$-ITRM-programs using any number of registers on a $\gamma$-ITRM with a fixed number of registers.
%\end{remark}

Towards the goal of this section, we recall the following result, along with its proof, from \cite{alpha itrms}, Proposition 69:

\begin{prop}{\label{strong successor}}
For each ordinal $\alpha$ and all $\alpha^{\prime}\in[\alpha+1,\alpha\omega)$, we have $\beta(\alpha+1)=\beta(\alpha^{\prime})$. 
Thus, for all $\gamma,\delta\in[\alpha+1,\alpha\omega)$, we have COMP$_{\gamma-\text{ITRM}}=$COMP$_{\delta-\text{ITRM}}$.\footnote{Note that this shows, conversely to the above footnote, how to ``trade'' increased register number for register capacity.} 
\end{prop}
%auskommentiert, den platz brauchen wir.
\begin{proof}
It is clear that $\beta(\alpha^{\prime})\geq\beta(\alpha+1)$, as, for $\gamma_{0}<\gamma_{1}$, $\gamma_{0}$-ITRMs can be simulated on $\gamma_{1}$-ITRMs. 

For the reverse direction, we recall the brief argument from \cite{alpha itrms} for the sake of the reader.

Suppose that $\alpha$ is a limit ordinal. Since $\alpha<\alpha^{\prime}<\alpha\omega$, there is $k\in\omega$ such that $\alpha^{\prime}<\alpha\cdot k$. 
We can thus simulate an $\alpha\cdot k$-ITRM on an $(\alpha+1)$-ITRM by replacing each register $R$ of the $\alpha\cdot k$-ITRM with $k$ registers $R_{1}$, ..., $R_{k}$ 
of the $(\alpha+1)$-ITRM, representing the ordinal $\gamma<\alpha\cdot k$ by writing $\gamma$ as $\alpha\cdot i+\rho$ with $i<k$ and $\rho<\alpha$ and 
then letting $R_{1}$, ..., $R_{i}$ contain $\alpha$, letting $R_{i+1}$ contain $\rho$ and letting $R_{j}$ contain $0$ for $j>(i+1)$. 

It follows that $\beta(\alpha)\leq\beta(\alpha^{\prime})\leq\beta(\alpha\cdot k)=\beta(\alpha)$, so that $\beta(\alpha)=\beta(\alpha^{\prime})$. 

If $\alpha$ is a successor ordinal, write $\alpha$ as $\hat{\alpha}+k$, where $\hat{\alpha}$ is limit ordinal and $k\geq 1$ is a natural number, so that $\alpha+1\geq\hat{\alpha}+2$. 
Then the above shows that, for $k\in\omega$ sufficiently large, $\beta(\hat{\alpha}+1)\leq\beta(\alpha+1)\leq\beta(\alpha^{\prime})\leq\beta(\hat{\alpha}\cdot k)=\beta(\hat{\alpha})$, so $\beta(\alpha+1)=\beta(\alpha^{\prime})$, as desired. 
\end{proof}

It thus remains to see that $\beta(\alpha)=\beta(\alpha+1)$. 
%neu: 14.03.2021

\begin{remark}
Note that Proposition \ref{strong successor} fails for unresetting machines: For example, given that $\alpha+1$-wITRMs can simulate $\alpha$-ITRMs, and using Theorem \ref{comp strength known}, we have 
COMP$_{\omega-\text{wITRM}}=\mathfrak{P}(\omega)\cap L_{\omega_{1}^{\text{CK}}}\subsetneq \mathfrak{P}(\omega)\cap L_{\omega_{\omega}^{\text{CK}}}=$COMP$_{\omega-\text{ITRM}}\subseteq$COMP$_{(\omega+1)-\text{wITRM}}$. 
\end{remark}

\begin{lemma}{\label{weak successor}}
For all ordinals $\alpha$, we have $\beta(\alpha)=\beta(\alpha+1)$. 
\end{lemma}

\begin{proof}
%To this end, l
If $\alpha=\bar{\alpha}+1$ is a successor ordinal, this follows from Proposition \ref{strong successor}: For then, applying Proposition \ref{strong successor} to $\bar{\alpha}$, we have $\beta(\alpha+1)=\beta((\bar{\alpha}+1)+1)=\beta(\bar{\alpha}+2)=\beta(\bar{\alpha})=\beta(\bar{\alpha}+1)=\beta(\alpha)$.

We can thus assume without loss of generality that $\alpha$ is a limit ordinal. 

Let $P$ be an $(\alpha+1)$-ITRM-program using the registers $R_{1}$, ..., $R_{n}$. We show how the actions of $P$ can be simulated on an $\alpha$-ITRM. Each register $R_{i}$ is represented by a triple $(\gamma,j,k)\in\alpha\times\{0,1\}\times\{0,1\}$. These triples are stored in three registers $(A_{i},F_{i},D_{i})$ (where $A$ stands for ``$\alpha$-part'', 
$F$ for ``final part'' and $D$ for ``detector''). The representation works via a decoding function $f$, which is defined by $f(\gamma,j,k)=\begin{cases}\gamma\text{, if }j=0\\\alpha\text{, if }j=1\end{cases}$. Note that this coding does not need the last component $k$; this will be used in the simulation to detect overflows at limit times: to this end, $D_{i}$ will contain $0$ when $A_{i}$ contains $0$ and otherwise, $D_{i}$ will contain $1$. We also use two flag registers to detect limit times.\footnote{See, e.g., \cite{C}, p. 17. We use two registers $R_{0}$, $R_{1}$ that initially contain $0$ and $1$ and swap their contents at each computation step. Thus, the computation will be at a limit time if and only if both registers contain $0$.}

We now explain how the simulation works at successor stages. Suppose that $P$ is run and the active program line is $L$. Then our simulation works as follows (where we abuse notation by confusing registers with their contents):

\begin{itemize}
\item If $L$ contains COPY($i$,$j$), then the contents of $A_{j}$, $F_{j}$ and $D_{j}$ are replaced by those of $A_{i}$, $F_{i}$ and $D_{i}$, respectively.
\item If $L$ contains $R_{i}\leftarrow R_{i}+1$, we distinguish two subcases: 

(i) If $F_{i}=0$, then leave $F_{i}$ and replace the content of $A_{i}$ by its successor (this will always work, since, by assumption, $\alpha$ is a limit ordinal). If $A_{i}=0$, additionally replace the content of $D_{i}$ by $1$. 

(ii) If $F_{i}=1$, replace the contents of $A_{i}$, $F_{i}$ and $D_{i}$ by $0$ (this corresponds to a reset due to an overflow).

\item If $L$ contains $R_{i}\leftarrow 0$, replace the contents of $A_{i}$, $F_{i}$ and $D_{i}$ by $0$. 
\item If $L$ contains ``IF $R_{i}=R_{j}$ THEN GOTO $l$'' and we have both $A_{i}=A_{j}$ and $F_{i}=F_{j}$, change the active program line to $l$. Otherwise, continue with the next program line.\footnote{This adapts to the comparison of finite sequences of registers in the obvious way.}
%\item If $L$ contains $R_{j}\leftarrow\mathcal{O}(i)$ and $i$ belongs to the oracle, replace  ORAKELAUFRUF VORERST AUSGELASSEN.
\end{itemize}

At a limit time $\delta$, it is not always the case that applying $f$ to the contents of $(A_{i},F_{i},D_{i})$ at time $\delta$ is the content of $R_{i}$ at time $\delta$. However, this content is easily calculated:
Namely, if $A_{i}\neq 0$ (that is, the inferior limit in $R_{i}$ was different from $0$, but below $\alpha$) or $A_{i}=D_{i}=0$ (i.e., if $R_{i}$ contained $0$ cofinally often), we leave the contents of $A_{i}$, $F_{i}$ and $D_{i}$ unchanged. If $A_{i}=0$ and $D_{i}\neq 0$, an overflow has taken place, and we replace the content of $F_{i}$ by $0$, leaving $A_{i}$ and $D_{i}$ unchanged at $0$. 

It is now easy to see that this simulation works as desired. %CHECK THIS! POSSIBLY WRITE MORE HERE!
\end{proof}

\begin{corollary}{\label{no change}}
%DAS KANN NICHT SEIN: Ist $L_{\alpha}\models\text{ZF}^{-}$, dann ist schon $\alpha+2$ deutlich st\"arker, weil man damit mindestens bis $L_{\alpha+2}$ rechnen kann. Oder? [Brauchen diese Konstruktionen starke Abschlusseigenschaften? -- Ja, f\"ur die Stapelverarbeitung. Es hat also noch eine Chance...]

For all ordinals $\alpha$ and all $\alpha^{\prime}\in[\alpha,\alpha\omega)$, we have $\beta(\alpha)=\beta(\alpha^{\prime})$. 
\end{corollary}
\begin{proof}
Immediate from Proposition \ref{strong successor} and Lemma \ref{weak successor}.
\end{proof}

\begin{question}
Given Corollary \ref{no change}, one might now conjecture that, in general, we have $\alpha^{j}=\alpha\omega$ for all $\alpha\in\text{On}$. Is this true? Note that, below, we will show that $\alpha^{\omega}\geq\alpha^{j}$ for certain values of $\alpha$. %This bound would be strict if we additionally knew that $\beta(\alpha)=\beta(\alpha^{2})$.
\end{question}

\section{The BH-dichotomy}

A particularly peculiar property of ITRMs is Theorem \ref{bounded halting}, i.e., the fact that, for any $k\in\omega$ the halting problem for ITRMs using at most $k$ registers is solvable by an ITRM-program (which, of course, uses more than $k$ registers); see Koepke and Miller \cite{KM}, Theorem $4$. From this, it is deduced in \cite{KM} that there is no universal ITRM. The argument relies crucially on properties of $\omega$ and does not generalize to any other multiplicatively closed ordinal.\footnote{In \cite{C}, Exercise 2.3.10, it is shown how the proof can be generalized to ordinals of the form $\omega\cdot k$ for $k\in\omega$.} In fact, we do not know whether any other such ordinals have this property. In this section, we will show that, for each $\alpha$, either $\alpha$ has the property just described, or there is, in a certain sense, a universal $\alpha$-ITRM. 

%Either there is a universal $\alpha$-ITRM or $\alpha$ satisfies the HB-property. [text aus mail einf\"ugen].

\begin{defini}
An ordinal $\alpha$ has the `bounded halting property' if and only if, for any $k\in\omega$, there are an $\alpha$-ITRM-program $P$ and an ordinal $\zeta<\alpha$ such that $P(\zeta)$ solves the halting problem for $\alpha$-ITRMs using at most $k$ registers. More specifically, for all $i\in\omega$, $\xi<\alpha$, $P(i,\xi,\zeta)$ halts with output $1$ if and only if, for the $i$-th program $P_{i}$ using at most $k$ registers, $P_{i}(\xi)$ halts and otherwise, it halts with output $0$.

If $\alpha$ has the bounded halting property, we also say that $\alpha$ is BH.
\end{defini}

\begin{defini}{\label{universal machine}}
If $P$ is an $\alpha$-ITRM-program and $\zeta<\alpha$, we say that $(P,\zeta)$ is $\alpha$-universal if and only if, for any $\alpha$-ITRM-computable set $x\subseteq\alpha$, 
%for all $\iota,\xi<\alpha$ and any $i\in\omega$, 
there are $j\in\omega$, $\nu<\alpha$ such that, for every $\iota<\alpha$, $P(j,\nu,\iota,\zeta)$ halts with output $1$ if and only if $\iota\in x$ and otherwise with output $0$.
%with the same output as $P_{i}(\iota,\xi)$ provided that $P_{i}(\iota,\xi)$ halts at all.
\end{defini}

\begin{thm}{\label{halting dichotomy}}
Let $\alpha$ be an ordinal. Then $\alpha$ is either BH or there are a program $P$ and $\zeta<\alpha$ such that $(P,\zeta)$ is $\alpha$-universal.
\end{thm}
\begin{proof} 
If $\alpha$ is a ZF$^{-}$-ordinal, then $\alpha$-ITRM-programs using $k$ registers halt before time $\alpha^{k+1}$; and moreover, it is easy to see that $\alpha^{k+1}$ is $\alpha$-ITRM-clockable for every $k\in\omega$. Consequently, the first alternative holds and the second fails. 

We thus assume from now on that $\alpha$ is ITRM-singular. In particular, by Lemma $34$ of \cite{alpha itrms}, we can carry out a well-foundedness check for orderings coded by subsets of $\alpha$ on an $\alpha$-ITRM.

Suppose that $\alpha$ is BH. We show that no $(P,\zeta)$ can be universal for $\alpha$. Suppose for a contradiction that $(P,\zeta)$ is universal for $\alpha$; let $P$ use $k$ many registers; moreover, pick an $\alpha$-ITRM-program $W$ that can test subsets of $\alpha$ for coding well-orderings and suppose that $W$ uses $l$ registers. Now, let $H$ be a program that solves the halting problem for $\alpha$-ITRMs with $k+l$ registers. Consider the following $\alpha$-ITRM-program $Q$: Run through all pairs $(i,\xi)\in\omega\times\alpha$ such that $i$ codes a program using at most $k$ registers and use $H$ to determine whether $P_{i}(\xi)$ will halt; if not, continue with the next pair, otherwise run $P_{i}(\xi)$ and then continue with the next pair. We show that this will halt after all $\alpha$-ITRM-halting times, which will be a contradiction. To see this, note that, by assumption all $\alpha$-ITRM-clockable ordinals are computable by $P$ for an appropriate choice of the parameters. But now, $Q$ in particular runs those programs that compute a code for some ordinal $\gamma$ and apply $W$ to this code, which takes at least $\gamma$ many steps. Since this happens for all clockable $\gamma$, the halting time of $Q$ will be above any $\alpha$-ITRM-clockable ordinal, a contradiction.

On the other hand, suppose that $\alpha$ is not BH. This implies that there is $k\in\omega$ such that the supremum of the ordinals clockable with an $\alpha$-ITRM using at most $k$ many registers is equal to $\beta(\alpha)$, for otherwise, we could use a program that halts after more steps as a stopwatch to solve the bounded halting problem for $k$ registers, contradicting the assumption that $\alpha$ is not BH. Pick $k$ such that this supremum attains $\beta(\alpha)$. It is not hard to see that there is a natural number $k^{\prime}$ such that, in fact, every ordinal $<\beta(\alpha)$ is clockable by an $\alpha$-ITRM-program using $k^{\prime}$ many registers: Namely, to clock $\xi<\alpha$, pick a program $Q$ using $k$ registers that halts in $\nu\geq\xi$ many steps. By Lemma $30$ in \cite{alpha itrms} (which generalizes a result from \cite{CFKMNW} on ITRMs), no configuration can occur in the halting computation of $Q$ at least $\omega^{\omega}$ many times. Thus, there is some ordinal $\xi^{\prime}<\omega^{\omega}$ and some $Q$-configuration $c$ such that, in the computation of $Q$, $c$ appears for the $\xi^{\prime}$-th time at time $\xi$. Using a few extra registers and given $c$ and $\xi^{\prime}$ as parameters, this can be detected and used to clock $\xi$. Moreover, again by a result from \cite{alphaitrms} which generalizes another result from \cite{CFKMNW}, $\alpha$-ITRM-computable ordinals have $\alpha$-ITRM-computable codes and in fact, the transition from a program $P$ that clocks $\xi$ to one that computes a code for $\xi$ is uniform in $P$ and uses a fixed number of extra registers depending only on the number of registers used by $P$. Thus, there is $k^{\prime\prime}\in\omega$ such that, for each $\xi<\beta(\alpha)$, there is an $\alpha$-ITRM-program using at most $k^{\prime\prime}$ many registers that computes a code for $\xi$. 

Now,  our universal program $\mathcal{U}$ will work as follows:  Let $P$ be an $\alpha$-ITRM-program that  computes some set $x\subseteq\alpha$. Pick $\xi<\beta(\alpha)$ minimal such that $x\in L_{\xi}$. By assumption, there must be an $\alpha$-ITRM-program $P^{\prime}$ that uses at most $k^{\prime\prime}$ many registers and computes a code for $\xi$. But now, by results in \cite{C} that generalize 
results in \cite{ORM}, there is an $\alpha$-ITRM-program $P_{L}$ that computes a code $d$ for $L_{\xi}$ from a code for $\xi$, uniformly in $\xi$. Finally, again by results in \cite{C}, a code $d$ for $L_{\xi}$ can be used to read out any $y\subseteq\alpha$ contained in $L_{\xi}$ when the index $\iota^{\prime}$ coding $y$ in the sense of $d$ is given. 

%On the other hand, suppose that $\alpha$ is not BH. This means that, for no $k\in\omega$, the supremum of the ordinals clockable with $\alpha$-ITRMs using at most $k$ many registers is equal to $\beta(\alpha)$, since we can use BH as in the last paragraph to obtain an $\alpha$-ITRM-program that runs for a longer time. Thus, for any $k\in\omega$, there is an $\alpha$-ITRM-program $L_{k}$ that halts, but its halting time is bigger than any halting time of an $\alpha$-ITRM with at most $k$ registers. But now, to solve the halting problem for $\alpha$-ITRMs with at most $k$ registers, we can simply use $L_{k}$ as a ``stopwatch''.
\end{proof}

\begin{remark}
We note that the above theorem yields a universal machine only in a rather weak sense of the word: Although we indeed obtain a program $U$ that is universal in the sense that, entering the appropriate parameters $i$ and $\gamma$, $U$ will compute the same function as $P_{i}(\gamma)$, this does work in any proper sense by $U$ simulating the work of $P_{i}$, but rather by computing the same function in a completely different way. This can be made precise by observing that there is no reason to expect $U$ to work relative to oracles $x\subseteq\alpha$. A more satisfying result would be that, for each $\alpha$ and $x\subseteq\alpha$, either the bounded halting problem for $\alpha$-ITRMs is solvable on $\alpha$-ITRMs \textit{uniformly in $x$} -- i.e., there is, for each $k\in\omega$, an $\alpha$-ITRM-program $H_{k}$ that solves the halting problem for $\alpha$-ITRMs with $k$ registers relative to every oracle $x\subseteq\alpha$ -- or a program $U$ such that, for all $x\subseteq\alpha$, all $i\in\omega$ and all $\gamma\in\alpha$, $U^{x}(i,\gamma)$ computes the same function as $P_{i}^{x}$. 
\end{remark}

\section{Restricted parameters}

Unless $\alpha=\omega$ or $L_{\alpha}\models\text{ZF}^{-}$, where the answer is positive, we are currently unable to answer whether the computational strength of $\alpha$-ITRMs increases with the number of registers admitted. It is natural to conjecture that this is the case in general.\footnote{We point out that, for ordinal register machines (ORMs), which have no bound on their register contents, there is indeed a universal machine with $10$ registers, see Koepke and Siders \cite{ORM1}. However, since overflows are ruled out for these machines, this seems to bear little analogy to $\alpha$-ITRMs.} Another natural stratification of the computational power of $\alpha$-ITRMs is the size of parameters: Restricting the initial register contents to elements of some $\gamma\leq\alpha$, and denoting by COMP$^{\gamma}_{\alpha-\text{ITRM}}$ the set of subsets of $\alpha$ thus computable, and by $\beta_{\gamma}(\alpha)$ the supremum of ordinals thus clockable (and by COMP$^{\gamma}_{\alpha-\text{wITRM}}$ and $\beta^{w}_{\gamma}(\alpha)$ the analogous concepts for $\alpha$-wITRMs), one would naturally expect that $\beta_{\gamma}(\alpha)$ keeps increasing with $\gamma$, at least for certain values of $\alpha$. However, it is not hard to see that these two natural conjectures contradict each other at least in certain cases: 

\begin{defini}
We say that $\alpha$ satisfies the bounded parameter property (BP) if and only if, for cofinally in $\alpha$ many $\gamma$, we have $\beta_{\gamma}(\alpha)>\text{sup}_{\iota<\gamma}\beta_{\iota}(\alpha)$. 

Moreover, for $k\in\omega$, let us denote by $\beta^{k}(\alpha)$ the supremum of ordinals clockable by $\alpha$-ITRMs (using arbitrary parameters) using at most $k$ many registers. 
\end{defini}

\begin{prop}
If $\text{cf}(\alpha)>\omega$, then the bounded halting property and the bounded parameter property cannot hold simultaneously for $\alpha$. 
\end{prop}
\begin{proof}
Suppose otherwise, so that $\alpha$ satisfies $\rho:=\text{cf}(\alpha)>\omega$ and $\alpha$ is both BH and BP. Clearly, we have $\beta(\alpha)=\text{sup}_{\iota<\alpha}\beta^{\iota}(\alpha)$ and $\beta(\alpha)=\text{sup}_{k\in\omega}\beta^{k}(\alpha)$. Define a function $f:\omega\rightarrow\rho$ by letting $f(k)$ be the smallest $\iota\in\rho$ such that $\beta_{\iota}(\alpha)>\beta^{k}(\alpha)$. Then $f[\omega]$ is unbounded in $\rho$, so $\text{cf}(\rho)=\omega$, contradicting the assumption.
%$f:\rho\rightarrow\omega$ by letting $f(\iota)$ be the smallest $k\in\omega$ such that $\beta^{k}(\alpha)>\beta_{\iota}(\alpha)$. Conversely, define $g:\omega\rightarrow\rho$ by $g(k):=\text{min}(f^{-1}[k])$. 
\end{proof}

We will now investigate the properties of $\alpha$-ITRMs with no or restricted parameters. This was considered in the case of $\alpha$-Turing machines by Rin in \cite{R}, and several of the questions considered below are motivated by Rin's work. 

We recall a %few 
standard definition. 

\begin{defini}
Let $\alpha$ be an ordinal. Then an ordinal $\tau$ is $\alpha$-ITRM-writable if and only if there are a bijection $f:\alpha\rightarrow\tau$ such that the set $\{p(\iota,\xi):\iota,\xi\in\alpha\wedge f(\iota)<f(\xi)\}$ is $\alpha$-ITRM-writable. The transfer to restricted parameters, $\alpha$-wITRMs etc. is straigthforward and we will not elaborate on it here. 

%Moreover, an ordinal $\alpha$ is called ITRM-singular if and only if there is an $\alpha$-ITRM-computable cofinal function $f:\beta\rightarrow\alpha$ with $\beta<\alpha$. 
\end{defini}

The following is a variant of Theorem $30$ of \cite{alpha itrms} (which, in turn, is a generalization of Lemma $4$ of \cite{CFKMNW}) for $\alpha$-ITRMs with restricted parameters.

\begin{lemma}{\label{clock write}}
For each ordinal $\alpha$, each $\gamma\leq\alpha$ and each ordinal $\tau$, if $\tau$ is $\alpha$-ITRM-clockable with parameters in $\gamma$, then $\tau$ is also $\alpha$-ITRM-writable with parameters in $\gamma$. 
\end{lemma}
\begin{proof} 
The proof of \cite{alpha itrms}, Theorem $30$, is easily seen to adapt to parameter restrictions. 
\end{proof}

\begin{remark}
Note, however, that the downards closure of the set of $\alpha$-ITRM-clockable ordinals, which is proved in \cite{alpha itrms}, Lemma $31$ (generalizing Lemma $3$ of \cite{CFKMNW}), does not necessarily continue hold when parameters are restricted: Thus, for example, it is easy to see that $\omega_{1}$ is clockable on an $\omega_{1}$-ITRM without parameters, but as long as only parameters contained in some $\gamma<\omega_{1}$ are admitted, there will only be countable (and thus boundedly) many $\omega_{1}$-ITRM-clockable ordinals below $\omega_{1}$. It is also easy to construct countable examples of this behavior via condensation arguments. 
\end{remark}

\begin{defini}
For $X\subseteq M$, let us write $\Sigma^{M}(X)$ for the $\Sigma_{1}$-Skolem hull of $X$ in $M$ under the canonical $\Sigma_{1}$-Skolem function for $L$ (in the language of set theory). 
\end{defini}

We need a slight strengthening of Lemma $33$ from \cite{alpha itrms}:

\begin{lemma}{\label{nice codes}}
If $\gamma$ is $\alpha$-ITRM-clockable and $\alpha$ is exponentially closed, then $\gamma$ has an $\alpha$-ITRM-computable code $c\subseteq\alpha$, such that the following functions are $\alpha$-ITRM-computable:

(i) the function that maps each $\iota<\alpha$ to the ordinal $\xi$ that it represents in $c$.

(ii) the function that maps $(0,\xi)$ with $\xi<\alpha$ to the ordinal $\iota$ which represents $\xi$ in the sense of $c$, and which also maps $(1,0)$ to the ordinal $\iota$ which represents  $(0,0)$ to the ordinal $\iota$ which represents $\alpha$ for each $\xi\leq\alpha$ the ordinal $\iota$ by which $\xi$ is represented in $c$.
\end{lemma}
\begin{proof}
In \cite{alpha itrms}, this was proved with (ii) restricted to the first case (i.e., $\xi<\alpha$). However, it is easy to see how to extend the argument to $\alpha$ itself: on input $(0,\xi)$, we proceed as in \cite{alpha itrms}, while on input $(1,0)$, we let the program $C$ that clocks $\gamma$ run for $\alpha$ many steps, thus compute the configuration of this program at time $\alpha$, and then run it again for $\alpha$ many steps to count how often this configuration occurs; the configuration at time $\alpha$ and the number of appearances of this configuration up to time $\alpha$ yield the desired index. To run $C$ for $\alpha$ many steps, simply run $C$ while counting upwards in some specified register until that register overflows. 
\end{proof}

\begin{thm}
Let $\gamma\leq\alpha\in\text{On}$, and suppose that $\alpha$ is exponentially closed.
%\begin{itemize}
%\item 
If $\alpha$ is ITRM-singular, then COMP$^{\gamma}_{\alpha-\text{ITRM}}=\Sigma_{1}^{L_{\beta_{\gamma}(\alpha)}}(\gamma\cup\{\alpha\})\cap\mathfrak{P}(\alpha)$.
%\item If $\alpha$ is wITRM-singular, then COMP$^{\gamma}_{\alpha-\text{wITRM}}=\Sigma_{1}^{L_{\beta^{2}_{\gamma}(\alpha)}}(\gamma)\cap\mathfrak{P}(\alpha)$. %[das braucht f\"ur die zweite Richtung evtl. wITRM-Singularit\"at, um $\Sigma_1$-Formeln im Code auswerten zu k\"onnen? $\Delta_0$ w\"urde reichen...]
%\end{itemize}
\end{thm}
\begin{proof}
%Both proofs work completely analogous. We thus only prove (1).
Let us write $H:=\Sigma_{1}^{L_{\beta_{\gamma}(\alpha)}}(\gamma\cup\{\alpha\})\cap\mathfrak{P}(\alpha)$.

%WAS IST OHNE ITRM-SINGULARITAET? DANN IST $\beta(\alpha)=\alpha^{\omega}=\beta_{0}(\alpha)$, wir k\"onnen vermutlich einen ``netten'' Code berechnen, in dem $\iota<\alpha$ durch $\omega\iota$ codiert wird und dann GEHT ES AUCH? DO IT!!!

First, suppose that $x\in$COMP$^{\gamma}_{\alpha-\text{ITRM}}$. Let $P$ be a program and $\rho<\gamma$ a parameter such that $P(\rho)$ computes $x$. Let $\tilde{P}(\xi)$ be the program that, successively for all $\iota<\alpha$, runs $P(\xi,\iota)$. Since $P(\rho)$ is still a program running in the parameter $\rho<\gamma$ and $P(\rho,\iota)$ halts for every $\iota<\alpha$ by assumption, $\tilde{P}(\rho)$ halts in less than $\beta_{\gamma}(\alpha)$ many steps; denote by $\tau$ the halting time of $\tilde{P}(\rho)$. Then $x$ is definable over $L_{\tau}$ and thus an element of $L_{\beta_{\gamma}(\alpha)}$. Moreover, the formula saying that $P(\rho)$ has a halting time is $\Sigma_{1}$ in the parameters $\rho$ and $\alpha$ over $L_{\beta_{\gamma}(\alpha)}$, and thus we have $\tau\in H$. Since satisfaction in $L_{\tau}$ for arbitrary $\in$-formulas is $\Delta_1$ in $\tau$, it follows that $x$ is $\Sigma_{1}$ over $L_{\beta(\gamma)}$ in $\iota<\gamma$ and $\alpha$, and thus $x\in H$. 

Conversely, assume that $x\in H$, and let $\phi$ be an $\Sigma_1$-formula and $\rho<\gamma$ such that $x$ is the $<_{L}$-minimal witness of $\phi(\rho,\alpha,y)$. Pick $\delta<\beta_{\gamma}(\alpha)$ minimal such that $L_{\delta}\models\exists{y}\phi(\rho,\alpha,y)$, so that $x\in L_{\delta}$. By definition of $\beta_{\gamma}(\alpha)$, there is an ordinal $\eta>\delta$ which is $\alpha$-ITRM-clockable with some parameter $\xi<\gamma$. By Lemma \ref{clock write}, $\eta$ is $\alpha$-ITRM-writable in parameters $<\gamma$, let $c$ be an $\alpha$-ITRM-writable code for $\alpha$ as in Lemma \ref{nice codes}. As in the proof of Lemma $34$ of \cite{alpha itrms}, we can now compute a code $d\subseteq\alpha$ for $L_{\eta}$ from $c$ in which, for each $\xi\leq\alpha$, the ordinal represented by $\iota$ in $c$ is represented by $\omega\xi$. From the parameter $\rho$, we can compute the ordinal $\rho^{\prime}<\alpha$ which codes $\rho$ in $d$; similarly, we can compute the ordinal $\alpha^{\prime}<\alpha$ which codes $\alpha$ in $d$.\footnote{This is the reason why Lemma $33$ from \cite{alpha itrms} had to be extended to include the search for the ordinal coding $\alpha$ itself: Although this is a single ordinal below $\alpha$, it may be larger than $\gamma$, so that we cannot use it as a parameter in the current context.}

 Now, $L_{\eta}\models\exists{y}\phi(\rho,\alpha,y)$ by upwards absoluteness of $\Sigma_{1}$-formulas and thus, $x\in L_{\eta}$. By searching exhaustively through $\alpha$, and using $\rho^{\prime}$ and $\alpha^{\prime}$, we can identify the ordinal $\iota\in\alpha$ that codes the $<_{L}$-minimal witness for $\phi(\rho,\alpha,y)$ (i.e., $x$). Using part (i) of Lemma \ref{nice codes}, it now follows that $x$ is $\alpha$-ITRM-computable from $c$ in the parameter $\rho$. 
%By Theorem $33$ of \cite{C1}, there is now an $\alpha$-ITRM-program
\end{proof}

\begin{remark}
Without the assumption of ITRM-singularity, this is clearly false, for then, we have $L_{\alpha}\models\text{ZF}^{-}$, so COMP$_{\alpha-\text{ITRM}}=\mathfrak{P}(\alpha)\cap L_{\alpha+1}$, while $L_{\beta_{0}(\alpha)}=L_{\beta(\alpha)}=L_{\alpha^{\omega}}$. Now, in $L_{\alpha^{\omega}}$, and using the parameter $\alpha$, we can easily define $L_{\alpha+2}$ and $L_{\alpha+1}$ by $\Sigma_{1}$-formulas and then use the $\Sigma_{1}$-formula ``There is an element of $L_{\alpha+2}$ which is not contained in $L_{\alpha+1}$'' to obtain an element of $H$ which is not $\alpha$-ITRM-computable.
\end{remark}

Since there are only countable many programs, it is clear that there are (many) values $\iota<\omega_{1}$ such that a parameter-free $\omega_{1}$-(w)ITRM cannot halt with $\iota$ in its first register. The same is true for every ordinal $\geq\omega_{1}$. Moreover, via condensation, the same result can be seen to hold for unboundedly in $\omega_{1}$ many countable ordinals.

\begin{question}
Determine the minimal ordinal $\alpha$ such that, for some $\iota<\alpha$, $\{\iota\}$ is not $\alpha$-ITRM-computable without parameters.\footnote{The analogous question for $\alpha$-ITTMs was considered and answered in \cite{CRS}.}
\end{question}

Dropping parameters has the effect that the set of clockable ordinals can have gaps.

\begin{lemma}

There exists an ordinal $\rho$ such that, for all $\mu<\rho$, there are ordinals $\alpha<\beta<\gamma$ such that $\alpha$ and $\gamma$ are halting times of $\rho$-ITRM-computations using only parameters $<\mu$, but $\beta$ is not.

\end{lemma}

\begin{proof}

Let $\rho=\omega_{1}^{L}$. Then the $L$-countable ordinals that are clockable by a $\rho$-ITRM with parameters less than $\mu$ have a countable supremum (since they form a countable set of countable ordinals); let us denote this supremum by $\xi(\mu)$. Now, every ordinal between $\xi(\mu)$ and $\omega_{1}^{L}$ will fail to be $\rho$-ITRM-clockable with parameters $<\mu$; but clearly, $\omega_{1}^{L}$ is $\rho$-clockable (e.g., by a program that counts upwards in some register until an overflow is detected). So we can let $\alpha=\omega$, $\beta=\xi(\mu)$ and $\gamma=\omega_{1}^{L}$.

\end{proof}

\begin{remark}

Countable examples of this phenomenon can be obtained by forming the elementary hull $H$ of the empty set in $L_{\omega_{2}}$, taking the transitive collapse $\overline{H}$ of $H$ and considering the image of $\omega_{1}$ under the collapsing map.

\end{remark}

%Contrast to Proposition 2.11 of Rin \cite{Rin}, parameter-free $\alpha$-ITRMs form a hierarchy with respect to real numbers:
%
% 
%
%\begin{thm}
%
%If $\alpha<\beta$, then COMP$^{0}_{\alpha-\text{ITRM}}\cap\mathfrak{P}(\omega)\subseteq$COMP$^{0}_{\beta-\text{ITRM}}\cap\mathfrak{P}(\omega)$ or
%
%COMP$^{0}_{\alpha-\text{ITRM}}\cap\mathfrak{P}(\omega)\supseteq$COMP^{0}_{\beta-\text{ITRM}}\cap\mathfrak{P}(\omega)$.
%
%\end{thm}
%
%\begin{proof}
%
%Assume that the second alternative fails, and let $x$ be a real number such that $x$ is computable by a parameter-free $\beta$-ITRM, but not by a parameter-free $\alpha$-ITRM. Let $P$ be a $\beta$-ITRM-program that computes $x$. By successively running $P(i)$ for all $i\in\omega$, we see that the supremum $\rho$ of the halting times of these computations is parameter-freely $\beta$-ITRM-clockable. It follows that subset of $\beta$ coding $\rho$ is $\beta$-ITRM-computable without parameters, and consequently, so is a code for $\rho+1$.
%
% 
%
%Ansatz: Wir kriegen jetzt einen reellen code f\"ur das $L$-level, in dem $x$ liegt, und also auch alle anderen reellen Zahlen aus diesem Level. Und jetzt?
%
%\end{proof}

In Rin \cite{R}, Theorem 2.8, it was shown that, for parameter-free $\alpha$-Turing machines, there exist values $\alpha$ and $\beta$ such that the computational strength of $\alpha$-Turing machines and $\beta$-Turing machines are incomparable with respect to the computable subsets of $\text{min}\{\alpha,\beta\}$ -- that is, none is a subset of the other.

We note here that the same obtains for register machines. The argument morally (i.e., with respect to the overall strategy) resembles those in \cite{R} and \cite{CRS}; however, the considerable differences between tape and register models -- such as the unavailability of a universal program that enumerates all computable sets -- require extra efforts. 

% [TO DO: sag etwas dar\"uber, wie die Argumente sich verhalten (diese sind subtiler, weil ITRMs weniger k\"onnen); schaue dir auchu noch den Artikel mit Philipp an, da gab es st\"arkere Ergebnisse)
The following lemma will come in handy.

\begin{defini}
Let $\alpha$ be an ordinal. A set $C\subseteq\mathfrak{P}(\alpha)$ is $\alpha$-ITRM-decidable if and only if there is an $\alpha$-ITRM-program $P$ and some parameter $\rho$ such that, for all $x\subseteq\alpha$, $P^{x}(\rho)$ halts with output $1$ if and only if $x\in C$, and otherwise, $P^{x}(\rho)$ halts with output $0$. %[TO DO: ORAKELSCHREIBWEISE AM ANFANG EINF\"UHREN!]
\end{defini}

\begin{lemma}{\label{parameter free index}}
Suppose that $\alpha$ is an exponentially closed ordinal and ITRM-singular\footnote{The ITRM-singularity is not necessary for the lemma to hold; however, the proof becomes somewhat more involved without this assumption and this is the only case that we will need.} ordinal and $C\subseteq\mathfrak{P}(\alpha)$ is $\alpha$-ITRM-decidable without parameters such that $(L_{\alpha+1}\setminus L_{\alpha})\cap C\neq\emptyset$. Then  $(L_{\alpha+1}\setminus L_{\alpha})\cap C$ contains a parameter-freely $\alpha$-ITRM-computable element. 
\end{lemma}
\begin{proof}
%Damit k\"onnen die entsprechenden Teile des n\"achsten und eines darauffolgenden Beweises durch Bezugnahme hierauf ersetzt werden. 

%Since $L_{\alpha+1}\setminus L_{\alpha}$ contains $ is an $\omega_{1}^{L}$-index, $L_{\alpha}$ is not a model of ZF$^{-}$ and thus ITRM-singular by \cite{alpha itrms} [STELLE!]. 
%Clearly, $\alpha$ is $\alpha$-ITRM-clockable without parameters by counting upwards in some register until that register contains $0$ and then making one further step; it follows from 
By \cite{alpha itrms}, Lemma $41$ that $L_{\alpha}$ has a parameter-freely $\alpha$-ITRM-computable code $c$ which is such that the function $f$ mapping $\iota<\alpha$ to the ordinal coding $\iota$ in the sense of $c$, along with its inverse function are $\alpha$-ITRM-computable without parameters.\footnote{Strictly speaking, the inverse function is a partial function. In the case that it is not defined, the algorithm is supposed to indicate this, e.g., by writing the value $1$ in some register specifically reserved for this purpose. Given that $f$ is $\alpha$--ITRM-computable, it is easy to see that this can be decided by simply checking, given $\xi<\alpha$, for each $\iota<\alpha$ whether $f(\xi)=\iota$.} 

%Moreover, we can identify the ordinal $\zeta<\alpha$ such that $\zeta$ codes $\omega_{1}^{L}$ in the sense of $c$ by searching through $\alpha$ and using $c$ to test, for each $\iota<\alpha$, whether $L_{\alpha}$ believes that $f(\iota)$ is the smallest uncountable cardinal. For this, we use the ability of $\alpha$-ITRMs to evaluate truth predicates, see \cite{alpha itrms} [STELLE (beschr\"ankte Evaluation reicht!)]. We can thus freely use the parameter $\zeta$ in the following. In particular, we can check, for each given $x\subseteq\alpha$, whether in fact $x\subseteq\omega_{1}^{L}$. 

Given $c$ and $f$ (along with its inverse), it is now possible to check, for each $x\subseteq\alpha$, whether or not $x\in L_{\alpha}$: To do this, run through $\alpha$, and, for each $\iota\in\alpha$, check whether $\iota$ happens to code $x$ in the sense of $c$. This, in turn, can be done by again searching through $\alpha$ and, for each $\xi<\alpha$, testing whether $\xi\in x\leftrightarrow p(f(\xi),\iota)\in c$ and also whether, if $p(\xi,\iota)\in c$, $\xi$ has a preimage under $f$. 

Pick $y\in (L_{\alpha+1}\setminus L_{\alpha})\cap\mathfrak{P}(\omega_{1}^{L})$. By \cite{alpha itrms}, Lemma $4$, there are an $\alpha$-ITRM-program $P$, some $n\in\omega$ and some $\rho<\alpha$ such that $P(\rho)$ computes $y$ in less than $\alpha^{n}$ many steps.\footnote{The time bound is left implicit in \cite{alpha itrms}, so a remark is in order how it is obtained: Roughly, $P$ works by evaluating truth in $L_{\alpha}$ for the formula $\phi$ defining $x$ over $L_{\alpha}$; if $\phi$ is $\Sigma_{n}$, this can be done by $n$ nested exhausted searches through $\alpha$, which can be done with time bound $\alpha^{n}$.} For $\iota<\alpha$, denote by $c(\iota)$ the set $\{\xi<\alpha:P(\rho,\xi)\downarrow=1\text{ in less than }\alpha^{n}\text{ many steps}\}$. Thus, $c(\rho)=y$. It is clear that $c(\iota)$ is $\alpha$-ITRM-computable uniformly in $\iota$, since $\alpha^{n}$ is easily seen to be $\alpha$-ITRM-clockable (without parameters) for any $n\in\omega$ (simply perform $n$ nested runs through $\alpha$ in $n$ separate registers). Moreover, let $Q$ be an $\alpha$-ITRM-program that decides $C$. 

Now, the desired program works like this: We count through $\alpha$ in some register. For every $\iota<\alpha$, we (i) use $c$ to check whether $c(\iota)\in L_{\alpha}$ and (ii) run $Q^{c(\iota)}$ (which is possible since $c(\iota)$ is uniformly $\alpha$-ITRM-computable from $\iota$). If the answer to (i) is positive or if  $Q^{c(\iota)}\downarrow=0$, we continue with $\iota+1$. Otherwise, a parameter $\xi$ has been identified such that $P$ computes a set $z\subseteq\alpha$ with the desired properties in the parameter $\xi$. Since $\xi$ was computed without the use of parameters, $z$ is parameter-freely $\alpha$-ITRM-computable. 
But it is clear that this will eventually happen, since, for $\iota=\rho$ at the latest, we obtain $c(\iota)=y$, which is as desired. 

%With time bound $\alpha^{n}$ (returning $0$ if the computation has not stopped by this time), $P$ computes some subset of $\alpha$ in the parameter $\iota$. Using the algorithms from the last two paragraphs, we can check (i) whether this is in fact a subset of $\omega_{1}^{L}$ and (ii) whether it is contained in $L_{\alpha}$. As soon as the answer to (i) is positive and the answer to (ii) is negative -- which will eventually happen, since, with $\iota=\rho$, we obtain $y$, which satisfies the condition -- we have identified some parameter $\iota$ such that $P$ computes an element of $(L_{\alpha+1}\setminus L_{\alpha})\cap\mathfrak{P}(\omega_{1}^{L})$ in the parameter $\iota$. Since this has been done without any initial parameters, the claim is proved. 
\end{proof}

\begin{lemma}{\label{incompatibility}}

[Cf. \cite{R}, Theorem 2.8] 

%There are ordinals $\alpha<\beta$ such that neither COMP$^{0}_{\alpha-\text{ITRM}}\subseteq$COMP$^{0}_{\beta-\text{ITRM}}\cap\mathfrak{P}(\alpha)$ nor COMP$^{0}_{\alpha-\text{ITRM}}\supseteq$COMP$^{0}_{\beta-\text{ITRM}}\cap\mathfrak{P}(\alpha)$.
For every ordinal $\mu$, there are ordinals $\alpha<\beta$ such that neither COMP$^{\mu}_{\alpha-\text{ITRM}}\subseteq$COMP$^{\mu}_{\beta-\text{ITRM}}\cap\mathfrak{P}(\alpha)$ nor COMP$^{\mu}_{\alpha-\text{ITRM}}\supseteq$COMP$^{\mu}_{\beta-\text{ITRM}}\cap\mathfrak{P}(\alpha)$.
\end{lemma}

\begin{proof}

%Let $\alpha$ be the supremum of the $L$-countable $\omega_{1}^{L}$-clockable ordinals, and let $\beta=\omega_{1}^{L}$. It is clear that $\alpha$ is countable, so that $\beta(\alpha)<\omega_{1}^{L}$. It is easy to see that $\alpha\in$COMP$^{0}_{\alpha-\text{ITRM}}\setminus$COMP^{0}_{\beta-\text{ITRM}}\cap\mathfrak{P}(\alpha)$.

We deal with the case $\mu=0$; the general case is an easy variation of this case. 

%Let us denote by $\sigma(\omega_{1}^{L})$ the smallest ordinal $\gamma>\omega_{1}^{L}$ such that $L_{\gamma}\prec_{\Sigma_{1}}L$. 
%Moreover, let 
Let us say that an ordinal $\gamma$ is an $\omega_{1}^{L}$-index if and only if $(L_{\gamma+1}\setminus L_{\gamma})\cap\mathfrak{P}(\omega_{1}^{L})\neq\emptyset$. Let us write $\mathcal{I}$ for the set of $\omega_{1}^{L}$-indices. 

\begin{claim}
For each $\zeta\in\mathcal{I}$, $\mathcal{I}\setminus\zeta$ has order type strictly greater than $\omega_{1}^{L}$. 
\end{claim}
\begin{proof}
To see this, let, for $\iota\leq\omega_{1}^{L}$, $\xi_{\iota}$ denote the minimal ordinal for which $L_{\xi_{\iota}}\models$``There are at least $\iota$ many ZF$^{-}$-ordinals greater than $\zeta$''. %$\omega_{1}^{L}$''. 
Clearly, we have $\xi_{\iota}\in\mathcal{I}$ for all such $\iota$, and moreover, we have $\xi_{\iota_{0}}<\xi_{\iota_{1}}$ whenever $\iota_{0}<\iota_{1}\leq\omega_{1}^{L}$. 
\end{proof}

\begin{claim}
For each $\xi\in\mathcal{I}$, we also have $\beta(\xi)+1\in I$. 
\end{claim}
\begin{proof}
Let $\xi\in\mathcal{I}$. Let $x\subseteq\omega_{1}^{L}$ be such that $x\in(L_{\xi+1}\setminus L_{\xi})$. Consider the $\Sigma_{1}$-formula that states ``There are ordinals $\alpha,\beta$ such that $x\in (L_{\alpha+1}\setminus L_{\alpha})$ and, every $\alpha$-ITRM-computation halts or loops in less than $\beta$ many steps and $L_{\beta}$ believes that, for every $\gamma$, there is an $\alpha$-ITRM-computation that halts in $\gamma$ many steps''. Clearly, $L_{\beta(\alpha)+1}$ is the minimal $L$-level in which this formula is true. By a standard fine-structural argument, it follows that $L_{\beta(\alpha)+1}$ is an $\omega_{1}^{L}$-index. 
\end{proof}

%Moreover, $\sigma(\omega_{1}^{L})$ is clearly closed under the map $\alpha\mapsto\beta(\alpha)$. For given $\alpha<\sigma(\omega_{1}^{L})$, we can consider the $\Sigma_{1}$-statement that states the existence of an ordinal $\beta$ such that, for all programs $P$ and all $\iota<\alpha$, $P(\iota)$ either halts or strongly loops in less than $\beta$ many steps. Since this statement is true in $L$, it must be true in $L_{\sigma(\omega_{1}^{L}}$, so that we must have $\beta(alpha)\in L_{\sigma(\omega_{1}^{L}}$.

\begin{claim}
If $\alpha$ is an $\omega_{1}^{L}$-index, then there is a parameter-freely $\alpha$-ITRM-computable set $c$ such that $c\in(L_{\alpha+1}\setminus L_{\alpha})\cap\mathfrak{P}(\omega_{1}^{L})$. 
\end{claim}
\begin{proof}

We use Lemma \ref{parameter free index}, where $C=\mathfrak{P}(\omega_{1}^{L})$. We thus need to check that the assumptions of this lemma are satisfied. 

First, since $\alpha$ is an $\omega_{1}^{L}$-index, $L_{\alpha}$ is not a model of ZF$^{-}$ and thus ITRM-singular by \cite{alpha itrms}, Lemma $24$.
 
%Clearly, $\alpha+1$ is $\alpha$-ITRM-clockable without parameters by counting upwards in some register until that register contains $0$ and then making one further step; it follows from \cite{alpha itrms} Lemma $41$ that $L_{\alpha}$ has a parameter-freely $\alpha$-ITRM-computable code $c$ which is such that the function $f$ mapping $\iota<\alpha$ to the ordinal coding $\iota$ in the sense of $c$, along with its inverse function are $\alpha$-ITRM-computable without parameters.\footnote{Strictly speaking, the inverse function is a partial function. In the case that it is not defined, the algorithm is supposed to indicate this, e.g., by writing the value $1$ in some register specifically reserved for this purpose. Given that $f$ is $\alpha$--ITRM-computable, it is easy to see that this can be decided by simply checking, given $\xi<\alpha$, for each $\iota<\alpha$ whether $f(\xi)=\iota$.} 

Second, we need to see that $\mathfrak{P}(\omega_{1}^{L})$ is $\alpha$-ITRM-decidable without parameters. As in the proof of Lemma \ref{parameter free index}, there is a parameter-freely $\alpha$-ITRM-computable code $c\subseteq\alpha$ for $L_{\alpha}$ such that the coding function $f$ and its inverse, the decoding function are parameter-freely $\alpha$-ITRM-computable. By Lemma $25$ of \cite{alpha itrms}, there is a parameter-free $\alpha$-ITRM-program $T$ that evaluates truth of $\in$-formulas in $L_{\alpha}$. Let $\phi$ be the formula ``$x$ is the smallest uncountable cardinal''. We can now run through $\alpha$ and use $T$ to check, for each $\iota<\alpha$, whether $f(\iota)$ is the smallest uncountable cardinal in $L_{\alpha}$. Since $\alpha>\omega_{1}^{L}$, the answer will eventually be positive, and then we will have found the unique $\zeta<\alpha$ which codes $\omega_{1}^{L}$ in the sense of $c$, i.e., such that $f(\omega_{1}^{L})=\zeta$. Now, to check whether a given set $x\subseteq\alpha$ is in fact a subset of $\omega_{1}^{L}$, run through $\alpha$ and check, for every $\iota<\alpha$, whether $\iota\in x\rightarrow p(f(\iota),\zeta)\in c$. If this is the case for all $\iota<\alpha$, we return $1$, otherwise, we return $0$. 

\end{proof}
%$c(\alpha)\subseteq\omega_{1}^{L}$ which codes $\alpha$. To see this... [DO IT: So ein $\alpha$ ist ITRM-singul\"ar. Code f\"ur $L_{\alpha}$ berechnen, dar\"uber ist so ein Code definierbar, die Elemente von $\omega_{1}^{L}$ sind identifizierbar durch den Isomorphietester.]

It is now easy to see that we cannot have COMP$^{0}_{\alpha-\text{ITRM}}\cap\mathfrak{P}(\omega_{1}^{L})\subseteq$COMP$^{0}_{\beta-\text{ITRM}}\cap\mathfrak{P}(\omega_{1}^{L})$ for all $\alpha,\beta\in\mathcal{I}$ with $\beta>\beta(\alpha)$: Otherwise, if $\eta$ is the $\omega_{1}^{L}$-th element of $\mathcal{I}$ -- which exists by the first claim -- COMP$^{0}_{\eta-\text{ITRM}}$ would have to contain some specific subset $y(\xi)\in(L_{\xi+1}\setminus L_{\xi})\cap\mathfrak{P}(\omega_{1}^{L})$ for each $\xi\in\mathcal{I}\upharpoonright\eta$, so we would have a constructible injection from $\omega_{1}^{L}$ into COMP$^{0}_{\eta-\text{ITRM}}$, contradicting the fact that, since there are only countable many programs, COMP$^{0}_{\eta-\text{ITRM}}$ is clearly countable in $L$. 

%Es kann nun nicht sein, dass COMP$^{0}_{\alpha-\text{ITRM}}\cap\mathfrak{P}(\omega_{1}^{L})\subseteq$COMP$^{0}_{\beta-\text{ITRM}}\cap\mathfrak{P}(\omega_{1}^{L})$ f\"ur jedes $\beta\in\mathcal{I}$ mit $\beta>\beta(\alpha)$.
%ein $\omega_{1}^{L}$-Code f\"ur jedes $\alpha\in\mathcal{I}$ f\"ur jedes $\beta\in\mathcal{I}$ mit $\beta>\beta(\alpha)$ parameterfrei $\beta$-ITRM-berechenbar ist: Denn damit ist auch $\{\alpha\}$ parameterfrei $\beta$-ITRM-berechenbar [wieder der Isomorphietrick: probiere Ordinalzahlen der Reihe nach durch...], also ist COMP$^{0}_{\alpha-\text{ITRM}}\subseteq$COMP$^{0}_{\beta-\text{ITRM}}$. 
%Betrachte nun das $\omega_{1}^{L}$-te Element $\eta$ von $\mathcal{I}$: COMP$^{0}_{\eta-\text{ITRM}}$ m\"usste f\"ur jeden fr\"uheren Index die entsprechende Teilmenge von $\omega_{1}^{L}$ enthalten, w\"are also \"uberabz\"ahlbar, ist aber nur abz\"ahlbar. 

It follows that there are $\alpha,\beta\in\mathcal{I}$ such that $\beta>\beta(\alpha)$ and COMP$^{0}_{\alpha-\text{ITRM}}\cap\mathfrak{P}(\omega_{1}^{L})\not\subseteq$COMP$^{0}_{\beta-\text{ITRM}}\cap\mathfrak{P}(\omega_{1}^{L})$. Thus, some parameter-freely $\alpha$-ITRM-computable subset of $\omega_{1}^{L}$ is not $\beta$-ITRM-computable. 

%Also existiert ein $\alpha\in\mathcal{I}$ so, dass ein $\beta>\beta(\alpha)$ mit $\beta\in\mathcal{I}$ existiert mit COMP$^{0}_{\alpha-\text{ITRM}}\not\subseteq$COMP$^{0}_{\beta-\text{ITRM}}$. 

However, by the third claim, there is a parameter-freely $\beta$-ITRM-computable subset $z$ of $\omega_{1}^{L}$ which is contained in $L_{\beta+1}\setminus L_{\beta}$. Since $\beta>\beta(\alpha)$ by assumption, we have in particular $z\notin L_{\beta(\alpha)}$, so $z$ is not $\alpha$-ITRM-computable (not even with parameters). Thus, we also have that some parameter-freely $\beta$-ITRM-computable subset of $\omega_{1}^{L}$ is not $\alpha$-ITRM-computable. 

%Nach dem obigen kann eine parameterfreie $\beta$-ITRM aber einen Teilmenge von $\omega_{1}^{L}$ berechnen, die $\beta$ codiert, und also gr\"o{\ss}er ist als $\beta(\alpha)$. So eine Teilmenge kann man mit einer $\alpha$-ITRM aber nicht berechnen (nicht einmal mit Parametern!), nach Definition von $\beta(\alpha)$. Also ist auch COMP$^{0}_{\beta-\text{ITRM}}\cap\mathfrak{P}(\omega_{1}^{L})\subseteq$COMP$^{0}_{\alpha-\text{ITRM}}\cap\mathfrak{P}(\omega_{1}^{L})$

Thus, the pair $(\alpha,\beta)$ is as desired.\footnote{Again, countable examples can be obtained from this by condensation arguments.}

For general $\mu$ we replace $\omega_{1}^{L}$ with the cardinal $L$-successor of $\mu$ in the above argument.
\end{proof}

However, when one only considers real numbers, the ordinals indeed form a linear hierarchy with respect to parameter-free ITRM-computability strength. This was proved for parameter-free $\alpha$-ITTMs in \cite{CRS}, Theorem 2.19; the argument for $\alpha$-ITRMs is essentially the same, but the limitations of register models -- such as the nonexistence of a universal machine -- lead to a few extra subtleties. 

\begin{thm}{\label{real comparable}}
[Cf. \cite{CRS}, Theorem 2.18] For each infinite ordinal $\alpha$, there is an ordinal $\rho_{0}(\alpha)$ such that COMP$^{0}_{\alpha-\text{ITRM}}=L_{\rho_{0}(\alpha)}\cap\mathfrak{P}(\omega)$. 
Consequently, the set $\{$COMP$^{0}_{\alpha}\cap\mathfrak{P}(\omega):\alpha\in\text{On}\}$\footnote{That this, in spite of being indexed with the class of ordinals, is a set rather than a proper class follows from the fact that it is clearly a subset of $\mathfrak{P}(\mathfrak{P}(\omega))$.} is linearly ordered by $\subseteq$. 
%For each two ordinals $\alpha$, $\beta$, we have COMP$^{0}_{\alpha-\text{ITRM}}\subseteq$COMP$^{0}_{\beta-\text{ITRM}}$ or COMP$^{0}_{\beta-\text{ITRM}}\subseteq$COMP$^{0}_{\alpha-\text{ITRM}}$
\end{thm}
\begin{proof}
It clearly suffices to prove the first claim. To this end, we will show that, for any $\alpha$ and any real number $x$, if $x$ is $\alpha$-ITRM-computable without parameters, then there is a parameter-freely $\alpha$-ITRM-computable real number $l(x)$ that codes an $L$-level $L_{\beta}\ni x$. Once this is achieved, the proof finishes as follows: If $y\in L_{\beta}\cap\mathfrak{P}(\omega)$, then there is a natural number $i$ that codes $y$ in the sense of $l(x)$. As a natural number, $i$ is parameter-freely computable (by a program that applies the successor operation $i$ many times). Now, to determine for an arbitrary $j\in\omega$ whether $j\in y$, we need to determine the natural number $n(k)$ that codes $j$ in the sense of $l(x)$ and then check whether $p(j,i)\in y$. Identifying $n(k)$ (uniformly in $k$) is easily seen to be possible already on an ITRM (for details, see \cite{CFKMNW})\footnote{To give a brief sketch: $n(0)$ can be identified as the only natural number that has no predecessor in the sense of $l(x)$ (i.e., $l(x)$ contains no element of the form $p(m,n(0))$). Then, recursively, $n(m+1)$ is the unique natural number which has $n(0),...,n(m)$ as its predecessors in the sense of $c$, and no others. In the same way, the natural number coding $\omega$ in the sense of $l(x)$ can also be identified.}, and thus on a parameter-free $\alpha$-ITRM for any $\alpha\geq\omega$. 

Now for the claim. Let $x\subseteq\omega$ be given, and let $P$ be an $\alpha$-ITRM-program that computes $x$ without parameters. By successively running $P(i)$ for all $i\in\omega$, we see that the supremum $\rho$ of the halting times of these computations is parameter-freely $\alpha$-ITRM-clockable; hence, so is $\rho+2$. Since $x$ is definable over $L_{\rho}$ as $\{i\in\omega:P(i)\downarrow=1\}$, we have $x\in L_{\rho+1}$. Moreover, $\rho+1$ is minimal such that $L_{\rho+1}$ believes in the existence of an ordinal $\rho$ such that $L_{\rho}\models\forall{i\in\omega}P(i)\downarrow$; hence, $\rho+1$ is an index and thus, by \cite{BP}, Theorem $1$, a real number $c$ coding $L_{\rho+1}$ is contained in $L_{\rho+2}$. The argument that $\alpha$-ITRM-clockable ordinals are also $\alpha$-ITRM-computable given in \cite{alpha itrms} (generalized from the one in \cite{CFKMNW}) makes no use of parameters and thus in fact shows that a subset of $\alpha$ coding $\rho+2$ is $\alpha$-ITRM-computable without parameters. 
%Since $\rho$ is minimal with the property that $L_{\rho}\models\forall{i\in\omega}P(i)\downarrow$, $\rho$ is an index
By Lemma $41$ of \cite{alpha itrms}, there is a parameter-freely $\alpha$-ITRM-computable subset $a$ of $\alpha$ that codes $L_{\rho+2}$ (again, the argument makes no use of parameters). Now, given $c$, we can search through $\alpha$ for some $\iota$ which codes, in the sense of $c$, a subset $r$ of $\omega$ that codes a transitive model of $V=L$ which contains $x$: Being a subset of $\omega$ can be established as sketched in the last footnote. Again using the footnote, one can then use $c$ and $\iota$ to decide, for a given $i\in\omega$, whether $i\in r$. Checking the other properties -- coding a transitive model of $V=L$ that contains $x$ -- can then be done even on an ITRM, since these can perform well-foundedness checks by section $3$ of \cite{KM}, evaluate truth predicates in coded structures and check whether coded structures contain a given real number by \cite{CFKMNW}, and thus on an $\alpha$-ITRM whenever $\alpha\geq\omega$. 
%
%-Compute $\alpha$-code for $L_{\rho+1}$ 
%
%-Identify $\omega$ and the natural numbers as described in the last footnote. 
%
%-Now search for a real number $r$ that encodes a transitive model of $V=L$ and contains $x$.
%
%-Translate $r$, again using the footnote. 

%Assume that the second alternative fails, and let $x$ be a real number such that $x$ is computable by a parameter-free $\beta$-ITRM, but not by a parameter-free $\alpha$-ITRM. Let $P$ be a $\beta$-ITRM-program that computes $x$. By successively running $P(i)$ for all $i\in\omega$, we see that the supremum $\rho$ of the halting times of these computations is parameter-freely $\beta$-ITRM-clockable. It follows that subset of $\beta$ coding $\rho$ is $\beta$-ITRM-computable without parameters, and consequently, so is a code for $\rho+1$.

%Ansatz: Wir kriegen jetzt einen reellen code f\"ur das $L$-level, in dem $x$ liegt, und also auch alle anderen reellen Zahlen aus diesem Level. Und jetzt?
\end{proof}

As observed in \cite{CRS}, Theorem 2.18(a) for parameter-free $\alpha$-Turing machines, it is not the case that greater ordinals also yield greater computability strength with respect to real numbers:

\begin{prop}
[Cf. \cite{CRS}, Theorem 2.18] There are ordinals $\omega<\alpha<\beta$ such that COMP$^{0}_{\alpha-\text{ITRM}}\cap\mathfrak{P}(\omega)\supsetneq$COMP$^{0}_{\beta-\text{ITRM}}\cap\mathfrak{P}(\omega)$. 
\end{prop}
\begin{proof}
Whenever $\alpha>\omega$, there is a parameter-free $\alpha$-ITRM program that halts with $\omega$ in its first register $R_{1}$: Using two auxiliar registers, increment $R_{1}$ by $1$ in every step, while the auxiliar registers initially contain $1$ and $0$, respectively, and swap their contents in every step. Halt when both of these registers contain $0$, which will happen at the first limit ordinal, i.e., $\omega$, which will then be the content of $R_{1}$. Using this, it is now easy to see that $\mathfrak{P}(\omega)$ is $\alpha$-ITRM-decidable for every $\alpha\geq\omega$. 

Now, whenever $\alpha$ is an index, i.e., such that $L_{\alpha+1}\setminus L_{\alpha}$ contains a real number, then $L_{\alpha}$ is not a model of ZF$^{-}$, so $\alpha$ is ITRM-singular by \cite{alpha itrms}, Lemma $24$. Hence, the assumptions of Lemma \ref{parameter free index} are satisfied, and it follows that $L_{\alpha+1}\setminus L_{\alpha}$ contains a parameter-freely $\alpha$-ITRM-computable real number. It is standard that index ordinals are unbounded in $\omega_{1}^{L}$. Given this, we cannot have COMP$^{0}_{\alpha-\text{ITRM}}\subseteq$COMP$^{0}_{\omega_{1}^{L}-\text{ITRM}}$ for all $\alpha<\omega_{1}^{L}$, since the latter set is still countable.\footnote{Again, countable counterexamples can now be obtained via condensation arguments.}

\end{proof}

\section{Cardinal-Recognizing ITRMs}

In \cite{Ha}, Habic considered a new way of conveying extra information to a transfinite computation by introducing ``cardinal-recognizing Infinite Time Turing Machines'' which assume a special inner state whenever the computation time reaches an infinite cardinal. This turned out to considerably increase the computational of ITTMs. 
%Habic introduced this for ITTMs, it changes considerably the computational strength. 
The same is true for Ordinal Turing Machines: for example, \cite{C}, Exercise 4.4.6 shows that, if the set of real numbers in the universe is closed under the sharp operator, then so is the set of real numbers computable by cardinal-recognizing pOTMs. The idea of cardinal recognition can easily be adapted to register models. In this section, we will show that, perhaps surprisingly, for ITRMs, the ability to recognize cardinals is sterile: It does not change the set of computable objects. For general values of $\alpha$, we will see an increase in computational power due to the ability to recognize cardinals is equivalent to the solvability of the bounded halting problem and thus forms a further characterization in the BH-dichotomy.

\begin{defini}
Let $X$ be a class of ordinals. An $X$-recognizing $\alpha$-ITRM works like an $\alpha$-ITRM with an extra ``detection'' register $R_{D}$ that behaves as follows: Whenever the current computation time is contained in $X$, the content of $R_{D}$ is changed to $0$.\footnote{Note that this has the effect that $R_{D}$ contains $0$ at all times that are limits of elements of $X$. For our purposes, this effect is welcome, as the cases of $X$ relevant for us are closed under limits. If one wanted to avoid this, one could change $0$ to $1$ in the definition.} Let us write UCard for the class of uncountable cardinals. $\text{UCard}$-recognizing $\alpha$-ITRMs will be called ``cardinal-recognizing $\alpha$-ITRMs''. 

%$X$-Recognizing $\alpha$-ITRM: Special ``detection'' register, at times in $X$, it is set to $0$, no matter what it contained before. Cardinal-recognizing $\alpha$-ITRM: $X=$Card, the class of cardinals. 

For an $\alpha$-ITRM-program $P$, we denote by $^{X}P$ the program run as an $X$-recognizing $\alpha$-ITRM, i.e., run with the modifications in the behaviour of $R_{D}$ just described.  

%For $k\in\omega$ and $x\subseteq\omega$, let us denote by $\mathcal{A}^{x}_{k}$ the class of ordinal multiples of $\omega_{k}^{\text{CK},x}$. 
Moreover, for $\delta\in\text{On}$, let us write $M_{\delta}:=\{\delta\iota:\iota\in\text{On}\}$ for the set of multiples of $\delta$. We will abbreviate $M_{\omega_{k}^{\text{CK},x}}$ by $\mathcal{A}_{k}^{x}$ for $x\subseteq\omega$ and $k\in\omega$. %[VEREINHEITLICHEN?]
\end{defini}

\begin{remark}
That we use UCard rather than Card has technical reasons: Otherwise, the first $\omega$ many steps would all be registered as cardinals, which is an unwanted behaviour that would lead to inconvenient special cases. It is easy to see that UCard-computations can be carried out on $\text{Card}$-recognizing $\alpha$-ITRMs by simply running for $\omega$ many ``empty'' steps before starting the ``actual'' computation, so that the modification is insubstantial.
\end{remark}

For ITTMs, it was observed Habic \cite{Ha} that cardinal-recognizing ITTMs can solve the halting problem for ITTMs. This is not true in general for $\alpha$-ITRMs. There is, however, a natural and useful variant, the proof for $\alpha$-ITRMs follows the same idea as in \cite{Ha}, and which we shall now sketch. % we merely sketch it. 

\begin{prop}{\label{cardinals and ITRMs}}
For any $x\subseteq\alpha$ and any $\alpha$-ITRM-program $P$, $P^{x}$ will either halt in less than $\text{card}(\alpha)^{+}$ many steps or not halt at all. 
\end{prop}
\begin{proof}
It is proved in \cite{alpha itrms}, Theorem $37$ that $\beta(\alpha)$ is strictly smaller than the next $\Pi_{3}$-reflecting ordinal after $\alpha$; this result relativizes to oracles. Now, the next $\Pi_{3}$-reflecting ordinals after $\alpha$ is clearly smaller than the cardinal successor of $\alpha$.\footnote{This, of course, is an overkill. Alternatively, one can, assuming that $P^{x}$ halts in $\tau$ many steps, form the $\Sigma_{1}$-elementary hull $H$ of $\alpha+1\cup\{x\}$ in $L_{\tau+\omega}$; $H$ will contain the halting computation $D$ of $P^{x}$, and we will have $\text{card}(H)\leq\text{card}(\alpha)\cdot\omega$, so the transitive collapse of $H$ will only contain elements of cardinality less than $\text{card}(\alpha)^{+}$, which will include $D$.}
\end{proof}

\begin{prop}{\label{cardinals and halting}}
For each $k\in\omega$, $\alpha$-cITRMs can solve the halting problem for $\alpha$-ITRMs using $k$ registers, uniformly in the oracle. 
\end{prop}
\begin{proof}
Let $P$ be an $\alpha$-ITRM-program using $k\in\omega$ many registers, and let $x\subseteq\alpha$. We use an extra register $C$. Our $\alpha$-cITRM-program now works as follows: Use $k$ registers to run $P^{x}$, while simultaneously incrementing $C$ by $1$ for each step in the computation of $P^{x}$. Once $C$ overflows (i.e., contains $0$, set the content of $R_{D}$ to $1$ and let $P^{x}$ run on until it either halts -- in which case we return ``yes'' -- or the next cardinal time is reached, in which case $P^{x}$ has run for at least $\text{card}(\alpha)^{+}$ many steps without halting and thus will never halt, so that we can return ``no''. 
%DO IT!!! Let the program run, simultaneously count upwards in a separate register, when that register is reset, wait until the next cardinal, if it has not halted until then, it won't. This is because an $\alpha$-ITRM will, in whatever oracle, either halt in less than $\text{card}(\alpha)^{+}$ many steps or not at all. 
\end{proof}

\begin{remark}
Note that this does \textbf{not} mean that $\alpha$-cITRMs can solve the halting problem for $\alpha$-ITRMs. In fact, as Theorem \ref{cITRMs} shows, this fails already for $\alpha=\omega$. 
\end{remark}

Recall from the folklore that a ``strong loop'' in an infinite computation is a partial computation in which the first and the last state agree, and all states in between were in all components (active program line and register contents) at least as large as at this state. It is easy to see (see, e.g., \cite{KM}) that the presence of a strong loop implies that the program is not halting. Moreover, again by \cite{KM}, a non-halting program will always eventually run into a strong loop. 

Moreover, recall from \cite{alpha itrms} that the ``looping time'' of an $\alpha$-ITRM-program $P$ (possibly in some parameter and some oracle) is the minimal time $\tau$ such that the computation of $P$ up to $\tau$ contains a strong loop. It was shown in \cite{alpha itrms} that $\beta(\alpha)$, the supremum of the $\alpha$-ITRM-clockable ordinals, is also the supremum of the $\alpha$-ITRM-looping times. 

\begin{defini}
For an ordinal $\alpha$, $k\in\omega$, denote by $\delta_{\alpha}(k)$ the supremum of the $\alpha$-ITRM-halting and looping times (with parameters) for programs using at most $k$ registers. 
\end{defini}

\begin{prop}{\label{divisibility prop}}
For all infinite ordinals $\alpha$ and all $k\in\omega$, $\delta_{\alpha}(k)^{\omega}$ is a common multiple of all $\alpha$-ITRM-halting and looping times for programs using at most $k$ registers. 
\end{prop}
\begin{proof}
Let $\mu$ be the halting or looping time of some program using at most $k$ registers. Then $\mu<\delta_{\alpha}(k)$ by definition of $\delta_{\alpha}(k)$, and so $\delta_{\alpha}(k)^{\omega}\leq\mu\cdot\delta_{\alpha}(k)^{\omega}\leq\delta_{\alpha}(k)\cdot\delta_{\alpha}(k)^{\omega}=\delta_{\alpha}(k)^{\omega}$. 
%Suppose otherwise. Thus, $\delta_{\alpha}(k)$ leaves a residue when divided either by some halting time $\mu$ or some looping time $\lambda$. 

%Consider the first case. So there is a program $P$ using $k$ registers with halting time $\mu$ such that $\delta_{\alpha}(k)=\mu\cdot\iota+\rho$, where $\rho<\mu$. [hinreichend hohe haltezeit w\"ahlen, dann das dahinter schalten, das f\"uhrt zu einem widerspruch, aber was machen wir bei loopzeiten?9

%Now for the second case. 
\end{proof}

The following ``pulldown'' strategy, here adapted to register machines, is basic in the analysis of cardinal-recognizing ITTMs as conducted, e.g., in Habic \cite{Ha}. 

\begin{lemma}{\label{pulldown lemma}}
\begin{enumerate}[label=\roman*]
\item Let $P$ be an ITRM-program using $k\in\omega$ many registers. Then, for all $i\in\omega$ and all $x\subseteq\omega$, $^{\text{UCard}}P^{x}(i)$  halts if and only if  %$^{M_{\omega_{k+1}^{\text{CK},x}}{x}_{k+1}}P^{x}(i)$%
$^{\mathcal{A}^{x}_{k+1}}P^{x}(i)$ halts and both computations will have the same output. 
\item More generally, let $\alpha$ be an ordinal, $k\in\omega$. %, and let $\delta(k)$ be some ordinal which is a common ordinal multiple of all halting and looping times of $\alpha$-ITRM-programs (with parameters) using at most $k$ registers.  
Then, for all $\iota,\rho\in\alpha$ and all $x\subseteq\alpha$, $^{\text{UCard}}P^{x}(\rho,\iota)$  halts if and only if  
$^{M_{\delta_{\alpha}(k)^{\omega}}}P^{x}(\rho,\iota)$ halts and both computations will have the same output. 
\end{enumerate}
\end{lemma}
\begin{proof}
\begin{enumerate}
\item 
We claim that, for all $\iota,\delta\in\text{On}$ with $\iota>1$ and each $\rho<\omega_{k+1}^{\text{CK}}$, the states of $^{\text{UCard}}P^{x}(i)$ at time $\aleph_{\iota}+\omega_{k+1}^{\text{CK,x}}\cdot\delta+\rho$ agrees with that of $^{\mathcal{A}^{x}_{k+1}}P^{x}(i)$ at time $\omega_{k+1}^{\text{CK},x}\cdot(\iota+\delta)+\rho$, which implies the claim: For it follows in particular that, if $^{\text{Card}}P^{x}(i)$ is in the halting state at time $\aleph_{\iota}+\omega_{k+1}^{\text{CK},x}\cdot\delta+\rho$ and has $r$ in the output register, then the same will hold for $^{\mathcal{A}^{x}_{k+1}}P^{x}(i)$ at time $\omega_{k+1}^{\text{CK},x}(\iota+\delta)+\rho$, and vice versa.

To prove the claim, it suffices to see that the claim holds for $\delta=\rho=0$; for, if the states of the first computation at time $\aleph_{\iota}$ agrees with that of the second at time $\omega_{k+1}^{\text{CK},x}\cdot\iota$, then, as long as $\tau<\aleph_{\iota+1}$ (so that no cardinal-recognizing steps take place in the meantime), the state of the first computation at time $\aleph_{\iota}+\tau$ will also agree with that of the second at time $\omega_{k+1}^{\text{CK},x}\cdot\iota+\tau$. However, we know from Koepke \cite{K1}, Theorem $9$ that an ITRM-computation in the oracle $x$ using $k$ registers either halts in $<\omega_{k+1}^{\text{CK},x}$ many steps or runs into a strong loop of length $\omega_{k+1}^{\text{CK},x}$, so that the states at times of the form $\omega_{k+1}^{\text{CK},x}\cdot\iota$ will all be the same. Since $\aleph_{\iota}$ is uncountable and $\omega_{k+1}^{\text{CK},x}$ is countable, $\aleph_{\iota}=\omega_{k+1}^{\text{CK},x}\cdot\aleph_{\iota}$ is a multiple of $\omega_{k+1}^{\text{CK},x}$, so our claim is established. 

%This is proved by induction on $\tau(\iota,\delta,\rho):=\omega_{k+1}^{\text{CK},x}\cdot(\iota+\delta)+\rho$. [DO IT!!!] Falls falsch, w\"ahle $\iota,\delta,\rho$ so, dass keine Gleichheit vorliegt und $\tau$ minimal ist. Falls $\tau$ Nachfolger: trivial. Andernfalls folgt es aus Limesregeln. 
\item The general claim follows by an analogous argument.
\end{enumerate}
\end{proof}

\begin{lemma}{\label{clock lemma}}
\begin{enumerate}[label=\roman*]
\item For each ITRM-program $P$ and each $k\in\omega$, there is an ITRM-program $\tilde{P}$ which, for every $x\subseteq\omega$ computes the same function as $^{\mathcal{A}^{x}_{k}}P$. 
\item More generally, if $\alpha$ is exponentially closed\footnote{We assume exponential closure to be able to rely on the results from \cite{alpha itrms} used in the proof below. The result likely holds up without this assumption, at the price of some extra complications in the proof.} and BH, then, for every $\alpha$-ITRM-program $P$, each $k\in\omega$ and each parameter $\rho<\alpha$, there is an $\alpha$-ITRM-program $\tilde{P}$ such that $\tilde{P}(\rho)$ computes the same function as $^{M_{\delta_{\alpha}(k)^{\omega}}}P(\rho)$.
\end{enumerate}
\end{lemma}
\begin{proof}
\item It was shown in Koepke \cite{K1}, Theorem 10, that the supremum of the ITRM-clockable ordinals is $\omega_{\omega}^{\text{CK}}$, Moreover, it was shown in \cite{CFKMNW}, Theorem $6$ that the ITRM-clockable ordinals do not have gaps, so that they are exactly the elements of $\omega_{\omega}^{\text{CK}}$. Consequently, $\omega_{k}^{\text{CK}}$ is ITRM-clockable for every $k\in\omega$. 
For fixed $k$, pick an ITRM-program $Q$ that clocks $\omega_{k}^{\text{CK}}$. Let the program $P$ be given. Let us denote the ``detection'' register of $P$ by $R_{D}$ and let it initially contain $1$. The desired program $\tilde{P}$ now works as follows: Run $P$ and $Q_{k}$ simultaneously, by alternately carrying out single steps. 
When $Q_{k}$ halts, set $R_{D}$ to $0$, then reset the registers used by $Q_{k}$ to $0$ and start $Q_{k}$ again in the initial state. It is clear that this will work as desired at tlimes of the form $\omega_{k}^{\text{CK}}\times(\iota+1)$. 
However, at limit times, $R_{D}$ will also contain $0$, simply by the liminf rule for the register contents. 

\item If $L_{\alpha}\models\text{ZF}^{-}$, we know from \cite{alpha itrms} that a program using at most $k$ many registers halts or strongly loops in less than $\alpha^{k+1}$ many steps, so we can simply replace $\omega_{k}^{\text{CK}}$ by $\alpha^{k}$ in the case $\alpha=\omega$. 

The other case of general claim follows by the same strategy, once we have demonstrated the existence of a program $Q_{k}$ that clocks $\delta_{\alpha}(k)^{\omega}$. By assumption, there is program $H_{2k+r}$ that solves the halting problem for $\alpha$-ITRM-programs using at most $2k+r$ registers, where $r$ will be specified below. We can use this to implement a program that clocks $\delta_{\alpha}(k)$ by running through all pairs $(i,\rho)\in\omega\times\alpha$, using $H_{2k+r}$ to decide whether the computation of the $i$th program $P_{i}$ using at most $k$ registers in the parameter $\rho$ will halt. If it does, we run it until it halts. If it does not, we use $k+1$ additional registers to run through all $k$-tuples $\tau$ of elements of $\alpha$ and additionally all natural numbers $t$; for each such tuple $\tau$ and each such $t$, there is a program $S$ that runs $P_{i}(\rho)$ and waits for a strong loop with initial (and final) configuration $(t,\tau)$ (i.e., active program line $t$ and register contents $\tau$). It is easy to see that this can be done with a fixed extra number $r$ of registers. If such a loop is found, $S$ halts. Using $H_{2k+r}$, we can check whether $S$ will halt. If it does not, we know that $(t,\tau)$ does not start a strong loop in the computation in question, so we continue with the next configuration. This will eventually terminate and reveal the starting configuration $(t_{0},\tau_{0})$ of such a strong loop. Once this is found, we run $P_{i}(\rho)$ until the configuration $(t_{0},\tau_{0})$ appears for the second time. 

The routine just described halts at a time after all programs using at most $k$ many registers have either halted or run into a strong loop, i.e., after time $\delta_{\alpha}(k)$. It follows that $\delta_{\alpha}(k)$ is $\alpha$-ITRM-clockable. We still need to argue, though, that $\delta_{\alpha}(k)^{\omega}$ is also clockable. To see this, note that, since $L_{\alpha}\not\models\text{ZF}^{-}$, it now follows from \cite{alpha itrms} (Theorem 35) that $\delta_{\alpha}(k)$ is $\alpha$-ITRM-writable. Moreover, it is easy to see that there is an $\alpha$-ITRM-program $P_{\text{multiply}}$ such that, when $b$ and $c$ are $\alpha$-codese of ordinals $\beta$, $\gamma$, then $P_{\text{multiply}}(a,b)$ computes a code for $\beta\gamma$.\footnote{This can be done as follows: For $\iota_{0},\iota_{1},\xi_{0},\xi_{1}$, let $p(p(\iota_{0},\iota_{1}),p(\xi_{0},\xi_{1}))=1$ if and only if $p(\iota_{0},\xi_{0})\in a$ or $\iota_{0}=\xi_{0}$ and $p(\iota_{1},\xi_{1})\in b$. This is clearly $\alpha$-ITRM-computable.} Combining these two observations, it is easy to obtain a program $P_{\text{exp}}$ that, on input $n\in\omega$, computes an $\alpha$-code for $\delta_{\alpha}(k)^{n}$. Combining these codes into one to form a code for the sum of all these finite powers, we obtain that $\delta_{\alpha}(k)^{\omega}$ is $\alpha$-ITRM-computable. By Lemma 34 of \cite{alpha itrms}, it finally follows that $\delta_{\alpha}(k)^{\omega}$ is also $\alpha$-ITRM-clockable.

% If DO IT!: Alles nacheinander laufen lassen. Dann der Reihe nach potenzen berechnen: $\delta_{\alpha}(k)$ is clockable, hence writable, hence $\delta_{\alpha}(k)^{\omega}$ is writable, so it is clockable... [das erfordert evtl. aber gewisse voraussetzungen, kein (zf-)-modell, exponentiellen abschluss etc.?]
\end{proof}

\begin{thm}{\label{cITRMs}}
The computational strength of cITRMs is equal to that of ITRMs, i.e. COMP$_{\text{cITRM}}=$COMP$_{\text{ITRM}}=\mathfrak{P}(\omega)\cap L_{\omega_{\omega}^{\text{CK}}}$. This relativizes to oracles. 
\end{thm}
\begin{proof}
Clearly, the ITRM-computable subsets of $\omega$ are also cITRM-computable. 
The other direction is now an easy consequence of Lemma \ref{pulldown lemma} and Lemma \ref{clock lemma}: Given a real number $x\subseteq\omega$ computable by the cITRM-program $P$, suppose that $P$ uses $k$ registers. By Lemma \ref{pulldown lemma}, $x$ is computable by $^{\mathcal{A}_{k+1}}P$ and so, by Lemma \ref{clock lemma}, by some ITRM-program. 
\end{proof}

%As a further consequence,
Similarly, we can extend the BH-dichotomy by a third criterion:

\begin{thm}{\label{extended bh}}
For each exponentially closed\footnote{Again, exponential closure is a technical convenience rather than a necessary assumption.} ordinal $\alpha$, the following are equivalent: 

\begin{enumerate}
\item There is no universal $\alpha$-ITRM (in the sense of Definition \ref{universal machine} above)
\item $\alpha$ is BH (i.e., for any $k\in\omega$, the halting problem for $\alpha$-ITRMs using $k$ registers is solvable by an $\alpha$-ITRM). 
\item The computational strength of $\alpha$-ITRMs is equal to that of $\alpha$-cITRMs. 
\end{enumerate}
\end{thm}
\begin{proof}
The equivalence of (1) and (2) is Theorem \ref{halting dichotomy}. We show that (2)$\Leftrightarrow$(2). % and that (1)$\Rightarrow$(3).

Assume (3). By Proposition \ref{cardinals and halting},  the bounded halting problem for $\alpha$-ITRMs can be solved on $\alpha$-cITRMs, for any number $k$ of registers. By assumption, the same is true on $\alpha$-ITRMs. Hence, we have (2). 

We now show that (2)$\Rightarrow$(3). Assume that $\alpha$ is BH, and let $x\subseteq\alpha$ be $\alpha$-cITRM-computable. Let $P$ be an $\alpha$-ITRM-program, $k$ the number of its registers, $\rho<\alpha$ a parameter such that $^{\text{UCard}}P(\rho)$ computes $x\subseteq\alpha$. By Lemma \ref{pulldown lemma}, $x$ is also computed by $^{M_{\delta_{\alpha}(k)}^{\omega}}P(\rho)$. By Lemma \ref{clock lemma} there is an $\alpha$-ITRM-program $\tilde{P}$ such that $\tilde{P}(\rho)$ computes the same function as $^{M_{\delta_{\alpha}(k)}^{\omega}}P(\rho)$. Thus $x$ is $\alpha$-ITRM-computable. 

\end{proof}

\begin{remark}
Note again that the proof of Corollary \ref{extended bh} does not yield the relativization to oracles. Whether or not Corollary \ref{extended bh} holds relative to oracles is currently open. 
\end{remark}

%\subsection{Weak OTMs}

\section{Iterations of $\alpha$-ITRM-computable operators}{Iterations of $\alpha$-ITRM-computable operators\footnote{This section is taken from our CiE 2022-paper \cite{C2}.}}{\label{iteration section}}

%each ITRM-computable operator can be iterated $\omega^{2}$ many times on an $\omega^{\omega}$-ITRM. 

%more generally, $\alpha$-ITRM-computable operators can be iterated on an $\alpha^{\omega}$-ITRM; more precisely: if $F$ is an $\alpha$-ITRM-computable function, there is an $\alpha^{\omega}$-ITRM $P$ such that, on input $i\in\omega$, $P$ computes the $i$-th iterate $F^{i}$ of $F$. 

In \cite{alpha itrms}, it was proved that, if $L_{\alpha}\models$ZF$^{-}$, then the supremum of the $\alpha$-ITRM-clockable ordinals is $\alpha^{\omega}$. This situation, however, is rather special, and it was still consistent with the results obtained in \cite{alpha itrms} that the following natural generalization of Koepke's result on the computational strength of ITRMs (see \cite{K1}) holds:

\begin{conjecture}
Let $\alpha$ be an exponentially closed ordinal. Unless $L_{\alpha}\models$ZF$^{-}$, we have $\beta(\alpha)=\alpha^{+\omega}$.
\end{conjecture}

We will now show that this conjecture fails dramatically even for the first exponentially closed ordinal $\varepsilon_{0}=\omega^{\omega^{\omega^{...}}}$ greater than $\omega$. In fact, we will show that already $\beta(\omega^{\omega})$ is way bigger than $\omega_{\omega}^{\text{CK}}$.

\begin{defini}
Let $\alpha$ be an ordinal. We say that $F:\mathfrak{P}(\alpha)\rightarrow\mathfrak{P}(\alpha)$ is $\alpha$-ITRM-computable if and only if there is an $\alpha$-ITRM-program $P$ (possibly using a parameter $\gamma<\alpha$) such that, for all $x\subseteq\alpha$ and all $\iota<\alpha$, we have $P^{x}(\iota)\downarrow=1$ if and only if $\iota\in F(x)$ and otherwise $P^{x}(\iota)\downarrow=0$. In this situation, we also say that $P$ computes $F$. 
\end{defini}

\begin{defini}
For each infinite ordinal $\alpha$, pick an $\alpha$-ITRM-computable bijection $p_{\alpha}:\alpha\times\alpha\rightarrow\alpha$.\footnote{Since the set of $\alpha$-ITRM-computable subsets of $\alpha\times\alpha$ is a superset of $L_{\alpha+1}$, so that such a bijection is guaranteed to exist.}

Let $\alpha$ be an infinite ordinal, and let $F:\mathfrak{P}(\alpha)\rightarrow\mathfrak{P}(\alpha)$, $x\subseteq\alpha$. We define the iteration of $F$ along $\alpha$ as follows:
\begin{enumerate}
\item $F^{0}(x)=x$
\item $F^{\iota+1}(x)=F(F^{\iota}(x))$. 
\item When $\delta\leq\alpha$ is a limit ordinal, then $F^{\delta}(x)=\{p_{\alpha}(\iota,\xi):\iota<\delta,\xi<\alpha,\xi\in F^{\iota}(x)\}$.
\end{enumerate}
 
In addition %, when $F^{\beta}$ is ITRM-computable, 
we also write $F^{\beta\cdot k}$ for $(F^{\beta})^{k}$. 

\end{defini}

\begin{lemma}{\label{non-uniform finite iteration lemma}}
Let $\alpha$ be an ordinal, and let $F:\mathfrak{P}(\alpha)\rightarrow\mathfrak{P}(\alpha)$ be an $\alpha$-ITRM-computable function and let $n\in\omega$. Then $F^{n}$, the $n$-th iteration of $F$, is $\alpha$-ITRM-computable.
\end{lemma}
\begin{proof}
We prove this by induction. For $n=1$, there is nothing to show. Let $Q$ be an $\alpha$-ITRM-program that computes $F$ and let $Q_{n}$ be an $\alpha$-ITRM-program that computes $F^{n}$. 
Then an $\alpha$-ITRM-program $Q_{n+1}$ for computing $F^{n+1}$ works as follows: Run $Q$. Whenever $Q$ makes an oracle call to ask whether $\iota\in$F$^{n}(x)$, run $Q_{n}$ to evaluate this claim. When $Q$ uses $r_{0}$ many registers and $Q_{n}$ uses $r_{1}$ many registers, 
this can be implemented on an $\alpha$-ITRM using $r_{0}+r_{1}$ many registers. 
\end{proof}

The above iteration technique yields a new program for every iteration index $n$. The key for our main result is Lemma \ref{iteration lemma}, a uniform version of Lemma \ref{non-uniform finite iteration lemma}, which is our next goal. 

The following lemma is a standard application of ordinal arithmetic; as a coding device in infinite computability, it was already used by Koepke in \cite{K1}.

\begin{lemma}{\label{cnf liminf}}
Let $\alpha$ be an ordinal, $\delta$ be a limit ordinal, $(\gamma_{\iota}:\iota<\delta)$ a sequence of ordinals such that $\gamma_{\iota}<\alpha$ for each $\iota<\delta$, and let $\rho$, $\eta$ be arbitrary ordinals. 
Then $\text{liminf}_{\iota<\delta}\alpha^{\eta+2}\cdot\rho+\alpha^{\eta}\cdot\gamma_{\iota}=\alpha^{\eta+2}\cdot\rho+\alpha^{\eta}\cdot\text{liminf}_{\iota<\delta}\gamma_{\iota}$.
\end{lemma}

\begin{defini}
Let $\alpha,\beta$ be ordinals. We say that $\alpha$ is exponentially closed up to $\beta$ if and only if, for all $\gamma<\alpha$ and all $\iota<\beta$, we have $\gamma^{\iota}<\alpha$. 
\end{defini}

The following crucial observation is similar in spirit to the iteration lemma for infinite time Blum-Shub-Smale machines, see \cite{CG}, Lemma $10$. 

\begin{lemma}{\label{iteration lemma}}
Let $\alpha$ be closed under ordinal multiplication, and let $F:\mathfrak{P}(\alpha)\rightarrow\mathfrak{P}(\alpha)$ be $\alpha$-ITRM-computable. 
Moreover, let $\eta\in\text{On}$ be closed under ordinal addition. %das wird gebraucht, damit das verdoppeln klappt, das wir wiederum brauchen, um die überträge zu bemerken. 
Then there is an $\alpha^{\eta}$-ITRM-program $P_{\text{iterate}}$ such that, for all $\iota<\eta$, $P_{\text{iterate}}^{x}(\iota)$ computes $F^{\iota}(x)$. More precisely, for all $\iota<\eta$, $\xi<\alpha$, we will have $P_{\text{iterate}}^{x}(\iota,\xi)\downarrow=1$ if and only if $\xi\in F^{\iota}(x)$ and $P_{\text{iterate}}^{x}(\iota,\xi)\downarrow=0$, otherwise. 
%Then there is an $\alpha^{\omega}$-ITRM-program $P_{\text{iterate}}$ such that, for all $\iota<\alpha$, $P^{x}(\iota)$ computes $F^{\iota}(x)$.
%for all $i\in\omega$, all $x\subseteq\alpha$ and all $\iota<\alpha$, we have that $P^{x}(i,\iota)\downarrow=\chi_{F^{i}(x)}(\iota)$. 
\end{lemma}
\begin{proof}
Let $P$ be an $\alpha$--ITRM-program that computes $F$. 
Suppose that $P$ uses $n$ registers $R_{1},...,R_{n}$. The program $P_{\text{iterate}}$ will use registers $R_{1}^{\prime},...,R_{n}^{\prime}$ for simulating the register contents of $P$, a register $L$ for storing active program lines %, a ``stage register'' $S$ for storing the current iteration depth <-auskommentiert; so einfach passt das nicht zu den limites, weil "höhere" stufen kleinere indizes haben; es ginge mit dem cnf-trick, also potenzen von \alpha addieren; aber die nötige information ist ja schon im programmzeilenregister gespeichert; man muss nur darauf achten, dass es keine zeile mit index 0 geben darf...
and various auxiliar registers that will not be mentioned explictly. 

%We describe this on $\omega$ first, generalize to $\alpha$ later on (replace each $\omega$ with $\alpha$ and see what happens...).

The rough idea is this: When $\delta<\eta$ is a limit ordinal, the question whether $\xi\in F^{\delta}(x)$ can be decided by writing $\xi$ as $\xi=p_{\alpha}(\xi_{0},\xi_{1})$ and then deciding whether $\xi_{1}\in F^{\xi_{0}}(x)$; we will have $\xi_{0}<\xi$. 
To compute $F^{\iota+1}(x)$ for a given $\iota<\alpha$, $P_{\text{iterate}}$ will run $P$ in the oracle $F^{\iota}(x)$. This may again call $P$ for a lower iterate etc. Since $\alpha$ is well-founded, however, the nesting depth will remain finite at all times. %Since $\alpha$ is well-founded, $F^{\iota}(x)$ can be computed by finitely many nested calls of $P$ in the oracle $x$. 
At any time of this computation, there will be a configuration $(l^{\iota},r_{1}^{\iota},...,r_{n}^{\iota})$ corresponding to the outermost run of $P$, along with finitely many configuration $(l^{\xi_{1}},r_{1}^{\xi_{1}},...,r_{n}^{\xi_{1}})$ corresponding to the first iteration etc., up to $(l^{0},r_{1}^{0},...,r_{n}^{0})$ for the top iteration which works on input $x$ directly. 
 The program $P_{\text{iterate}}$ will store this by having $\alpha^{\iota\cdot 2}\cdot l^{\iota}+\alpha^{\xi_{1}\cdot 2}\cdot l^{1}+...+\alpha^{0}\cdot l^{0}$ in $L$ and $\alpha^{\iota\cdot 2}\cdot r_{i}^{\iota}+...+\alpha^{0}\cdot r_{i}^{0}$ in $R_{i}^{\prime}$. When the topmost computation terminates, it is taken off the stack and the computation ``below'' it is continued.

%The rough is idea is this: To compute $F^{k}(x)$ for a given $k\in\omega$, $P_{\text{iterate}}$ will run $P$ in the oracle $F^{k-1}(x)$; that is, it will run $P$, and, when an oracle call occurs, it will call $P$ again, but now with $F^{k-2}(x)$. Thus, in the oracle $x$, $F^{k}(x)$ can be computed by $k$ nested calls of $P$. 
%At any time of this computation, there will be a configuration $(l^{0},r_{1}^{0},...,r_{n}^{0})$ corresponding to the outermost run of $P$, a configuration $(l^{1},r_{1}^{1},...,r_{n}^{1})$ corresponding to the first iteration etc., up to $l^{k-1},r_{1}^{k-1},...,r_{n}^{k-1})$ for the top iteration which works on input $x$ directly. 
% The program $P_{\text{iterate}}$ will store this by having $\omega^{2k}\cdot l^{0}+\omega^{2k-2}\cdot l^{1}+...+\omega^{2}\cdot l^{k}$ in $L$ and $\omega^{2k}\cdot r_{i}^{0}+...+\omega^{2}\cdot r_{i}^{k}$ in $R_{i}^{\prime}$. When the innermost computation terminates, it is taken off the stack and the computation ``below'' it is continued. 

We now do it precisely. Suppose that $x\subseteq\alpha$ is given in the oracle, and that some ordinal $\iota<\eta$ is given in the first register. Our goal is to compute $F^{\iota}(x)$. 

The computation proceeds in $\iota+1$ many ``levels'', where a computation step takes place at level $\xi\leq\iota$ when it belongs to an evaluation of $F^{\xi}$. When an oracle call of the form $\mathcal{O}(\zeta)$ is made in level $\xi+1$, the computation enters level $\xi$; when it takes place in level $\delta$ with $\delta$ a limit ordinal and $\zeta$ is of the form $p_{\alpha}(\zeta_{0},\zeta_{1})$, the computation continues at level $\zeta_{0}$ with the computation of $F^{\zeta_{0}}(\zeta_{1})$. For the sake of convenience, we use a register $S$ for storing the sequence $(\xi_1,...,\xi_k)$ of currently relevant levels in the form $\alpha^{2\xi_1}+...+\alpha^{2\xi_k}$, where, of course $\xi_1>\xi_2>...>\xi_k$. 
%Computation in levels. $0$th level is a computation by $P$, everything is stored after multiplication with $\omega^{2k}$. When an oracle call of the form $\mathcal{O}(\iota)$ is made in level $i$, the computation enters level $(i+1)$. On level $k$, an oracle call $\mathcal{O}(\iota)$ is simply evaluated by checking whether $\iota\in x$. We store the level in a separate register $S$.

We now describe how to carry out instructions at level $\delta\leq\iota$ (all contents of registers other than the ones explicitly mentioned are left unchanged). Note that $\delta$ can be reconstructed from the content of the line register $L$, the content of which will be of the form $\alpha^{\gamma\cdot 2}\cdot\rho+\alpha^{\delta\cdot 2}\cdot l$ with $\gamma>\delta$ and $l>0$ (since, as we recall from the introduction, we start the enumeration of program lines with $1$). The $l$ appearing as the coefficient in this representation will be the index of a program line of $P$; depending on the content of this program line, the following steps are carried out: 

\begin{itemize}
\item (Before carrying out the other steps:) When $R_{i}$ contains an ordinal of the form $\alpha^{\gamma\cdot 2}\cdot\rho+\alpha^{\delta\cdot 2+1}$ for any $i\leq n$, replace it with $\alpha^{\gamma\cdot 2}\cdot\rho$ (this corresponds to a reset after a register overflow). 

\item The active program line contains the command $R_{i}\leftarrow R_{i}+1$: Read out the content of $R_{i}^{\prime}$. It will be an ordinal of the form $\alpha^{\gamma\cdot 2}\cdot\rho+\alpha^{\delta\cdot 2}\cdot r_{i}$ with $\gamma>\delta$; replace it with $\alpha^{\gamma\cdot 2}\cdot\rho+\alpha^{\delta\cdot 2}\cdot(r_{i}+1)$. 
Moreover, the content of $L$ will be an ordinal of the form $\alpha^{\gamma\cdot 2}\cdot\rho^{\prime}+\alpha^{\delta\cdot 2}\cdot l$; replace it with $\alpha^{\gamma\cdot 2}\cdot\rho^{\prime}+\alpha^{\delta\cdot 2}\cdot (l+1)$.

%BIS HIER VERALLGEMEINERT.
\item The active program line contains the command COPY$(i,j)$: Read out the contents of $R_{i}$ and $R_{j}$, which will be of the forms $\alpha^{\gamma_{0}\cdot 2}\cdot\rho+\alpha^{\delta\cdot 2}\cdot r_{i}$ and $\alpha^{\gamma_{1}\cdot 2}\cdot\rho^{\prime}+\alpha^{\delta\cdot 2}\cdot r_{j}$, where $\delta<\gamma_{0},\gamma_{1}$. 
Replace the content of $R_{i}$ with $\alpha^{\gamma_{0}\cdot 2}\cdot\rho+\alpha^{\delta\cdot 2}\cdot r_{j}$; modify the content of $L$ as in the incrementation operation.

\item The active program line contains the command IF $R_{i}=R_{j}$ GOTO $l$: Read out the contents of $R_{i}$ and $R_{j}$, which will be of the forms $\alpha^{\gamma_{0}\cdot 2}\cdot\rho+\alpha^{\delta\cdot 2}\cdot r_{i}$ and $\alpha^{\gamma_{1}\cdot 2}\cdot\rho^{\prime}+\alpha^{\delta\cdot 2}\cdot r_{j}$, where $\gamma_{0},\gamma_{1}>\delta$; moreover, 
let $\alpha^{\gamma_{2}\cdot 2}\cdot\rho^{\prime\prime}+\alpha^{\delta\cdot 2}\cdot l^{\prime}$ be the content of $L$, where $\delta<\gamma_{2}$. If $r_{i}=r_{j}$, replace the content of $L$ with $\alpha^{\gamma_{2}\cdot 2}\cdot\rho^{\prime\prime}+\alpha^{\delta\cdot 2}\cdot l$; if not, replace 
it with $\alpha^{\gamma_{2}\cdot 2}\cdot\rho^{\prime\prime}+\alpha^{\delta\cdot 2}\cdot (l^{\prime}+1)$.

\item The active program line contains the oracle call $\mathcal{O}(\xi)$ and $\delta=\bar{\delta}+1<\iota$ is a successor ordinal:  Let $\alpha^{\gamma_{0}\cdot 2}\cdot\rho+\alpha^{\delta\cdot 2}\cdot r$ be the content of $R_{1}$, and let $\alpha^{\gamma_{1}\cdot 2}\cdot\rho^{\prime}+\alpha^{\delta\cdot 2}\cdot l$ be the content of $L$, where $\delta<\gamma_{0},\gamma_{1}$. 
Replace the content of $R_{1}$ by $\alpha^{\gamma_{0}\cdot 2}\cdot\rho+\alpha^{\delta\cdot 2}\cdot r+\alpha^{\bar{\delta}\cdot 2}\cdot \xi$ and replace the content of $L$ by $\alpha^{\gamma_{1}\cdot 2}\cdot\rho^{\prime\prime}+\alpha^{\delta\cdot 2}\cdot l+\alpha^{\bar{\delta}\cdot 2}\cdot 1$. %Increase the content of $S$ by $1$.
% Modify the content of $L$ as in the incrementation operation. 
Also, we are now working at level $\bar{\delta}$, so we add $\alpha^{\bar{\delta}\cdot 2}$ to the content of $S$. 

\item  The active program line contains the oracle call $\mathcal{O}(\xi)$ and $\delta<\iota$ is a limit ordinal: Calculate $\xi_{0}$, $\xi_{1}$ with $\xi=p_{\alpha}(\xi_{0},\xi_{1})$. If $\xi_{0}\geq\delta$, return $0$ and modify the content of $L$ as in the incrementation operation. (Note that this output will be right due to the definition of the iteration at limit levels). 
If $\xi_{0}<\delta$, we need to check whether $\xi_{1}\in F^{\xi_{0}}(x)$. 
 The computation will then enter level $\xi_{0}$. Thus, we add $\alpha^{\xi_{0}\cdot 2}$ to the content of $S$.  Let $\alpha^{\gamma_{0}\cdot 2}\cdot\rho+\alpha^{\delta\cdot 2}\cdot r$ be the content of $R_{1}$, and let $\alpha^{\gamma_{1}\cdot 2}\cdot\rho^{\prime}+\alpha^{\delta\cdot 2}\cdot l$ be the content of $L$, where $\delta<\gamma_{0},\gamma_{1}$. 
Replace the content of $R_{1}$ by $\alpha^{\gamma_{0}\cdot 2}\cdot\rho+\alpha^{\delta\cdot 2}\cdot r+\alpha^{\xi_{0}\cdot 2}\cdot\xi_{1}$ and replace the content of $L$ by $\alpha^{\gamma_{1}\cdot 2}\cdot\rho^{\prime\prime}+\alpha^{\delta\cdot 2}\cdot l+\alpha^{\xi_{0}\cdot 2}\cdot 1$.

\item  The active program line contains the oracle call $\mathcal{O}(\xi)$ and $\delta=0$: This means that we are simply making a call to the given oracle, with no iterations of $F$ applied to it. Let $\alpha^{\gamma_{0}\cdot 2}\cdot\rho+\alpha^{\delta\cdot 2}\cdot r$ be the content of $R_{1}$. Check whether $\xi\in x$ (recall that $x$ is our oracle). If yes, replace the content of $R_{1}$ by $\alpha^{\gamma_{0}\cdot 2}\cdot\rho+\alpha^{\delta\cdot 2}\cdot 1$, otherwise, 
replace the content of $R_{1}$ by $\alpha^{\gamma_{0}\cdot 2}\cdot\rho$.  Modify the content of $L$ as in the incrementation operation.

\item When the coefficient of the minimal power of $\alpha$ in the Cantor normal form representation of the content of $L$ is the index of a line of $P$ that contains the ``halt'' command: 
Let $R_{1}$ contain $\alpha^{\gamma_{0}\cdot 2}\cdot\rho^{\prime}+\alpha^{\gamma_{1}\cdot 2}\cdot r+\alpha^{\delta\cdot 2}\cdot r^{\prime}$; replace it with $\alpha^{\gamma_{0}\cdot 2}\cdot\rho^{\prime}+\alpha^{\gamma_{1}\cdot 2}\cdot r^{\prime}$ (the result of the oracle call is passed down to the level that made the call).

For $i\in\{2,...,n\}$, let $R_{i}$ contain $\alpha^{\gamma_{0,i}\cdot 2}\cdot\rho_{i}+\alpha^{\gamma_{1,i}\cdot 2}\cdot r_{i}+\alpha^{\delta\cdot 2}\cdot r_{i}^{\prime}$; replace it with 
$\alpha^{\gamma_{0,i}\cdot 2}\cdot\rho_{i}+\alpha^{\gamma_{1,i}\cdot 2}\cdot r_{i}$ (the topmost layer corresponding to the now finished computation is deleted).

Also, if the content of $S$ is $\alpha^{\nu}\cdot\rho+\alpha^{\delta\cdot 2}$, replace it with $\alpha^{\nu}\cdot\rho$ (the last entry in the sequence of currently relevant levels is deleted). 

Finally, let the content of $L$ be $\alpha^{\gamma_{0}\cdot 2}\cdot\rho^{\prime\prime}+\alpha^{\gamma_{1}\cdot 2}\cdot l+\alpha^{\delta\cdot 2}\cdot l^{\prime}$; replace it by $\alpha^{\gamma_{0}\cdot 2}\cdot\rho^{\prime\prime}+\alpha^{\gamma_{1}\cdot 2}\cdot (l+1)$ (the active program line is increased by $1$, as the oracle command has been carried out). %Decrease the content of $S$ by $1$. 
\end{itemize}

$P_{\text{iterate}}$ now works on input $(\iota,\xi)\in\eta\times\alpha$ by first instantiating $L$ with $\alpha^{\iota\cdot 2}$, $R_{1}$ with $\alpha^{\iota\cdot 2}\cdot\xi$ and $R_{i}$ with $0$ for $i\in\{2,3,...,n\}$ and then carrying out the above instructions. By additive closure of $\eta$, we will have $\gamma\cdot 2<\eta$ whenever $\gamma<\eta$, so that all register contents generated in this procedure will be below $\alpha^{\eta}$. By induction on $\iota$ and using Lemma \ref{cnf liminf}, the program works as desired. %MORE HERE.
\end{proof}

%This lemma can be strengthened:

We note some important consequences of this result:

%unify this with the preceding lemma/proof
\begin{corollary}{\label{full iteration lemma}}
Let $\alpha>\omega$ be exponentially closed, and let $\beta<\alpha$. Moreover, let $F:\mathfrak{P}(\beta)\rightarrow\mathfrak{P}(\beta)$ be a $\beta$-ITRM-computable operator. Then:
\begin{enumerate}
\item %For each $\iota<\alpha$,  $F^{\iota}$ of $F$ is $\alpha$-ITRM-computable. 
 There is an $\alpha$-ITRM-program $P$ such that, for each $x\subseteq\beta$ and each $\iota<\alpha$, $P^{x}(\iota)$ computes $F^{\iota}(x)$. 
\item $F^{\alpha}$, the $\alpha$-th iteration of $F$, is $\alpha$-ITRM-computable.
\item $F^{\alpha\cdot i}$, the $\alpha\cdot i$-th iteration of $F$, is $\alpha$-ITRM-computable, for every $i\in\omega$. 
\end{enumerate}
\end{corollary}
\begin{proof}
\begin{enumerate}
\item  Since $\beta^{\iota+1}<\alpha$ by exponential closure of $\alpha$, is a direct consequence of Lemma \ref{iteration lemma}. 
%Same proof as for lemma \ref{iteration lemma}, but use basis $\alpha$ and represent computations of level $\xi\leq\iota$ as multiples of $\beta^{\xi}$. 
%More precisely: The computation of $F^{\iota}$ takes place at level $\iota$ and is represented by multiples of $\beta^{\iota}$. 
%If this requires a bit of $F^{\xi}$, $\xi<\iota$, we represent the corresponding computation by multiples of $\beta^{\xi}$ and so on. Since only finitely many levels can be involved at any specific time due to $\iota$ being well-founded, this works.

%THIS DOESN'T WORK: there are no infinite descending sequences in $\alpha$, therefore, the analogue of the above $\omega$-iteration does not work, as we cannot assign decreasing exponents to decreasing iteration stages. 
%(at level 0, we would like to have $\alpha^{\iota}$ for the $\iota$-iteration DOCH, DAS GEHT, wie beschrieben. 

\item In order to decide whether $p_{\alpha}(\xi_{0},\xi_{1})\in F^{\alpha}(x)$, use the algorithm $P$ from (1) to decide whether or not $\xi_{1}\in F^{\xi_{0}}(x)$. 
%The iteration in (1) is uniform, just do it successively for every $\iota<\alpha$. 

\item This is a consequence of (2) and Lemma \ref{non-uniform finite iteration lemma}. 
\end{enumerate}
\end{proof}

We now extract information on $\beta(\alpha)$, for various values of $\alpha$, thus, in particular, refuting the conjecture mentioned above that $\beta(\alpha)=\alpha^{+\omega}$ unless $L_{\alpha}\models\text{ZF}^{-}$. 

\begin{defini}
Let $\alpha$ be an ordinal. By recursion, we define, for $\iota\in\text{On}$: $^{0}\alpha=\alpha$, $^{\iota+1}\alpha=\alpha^{^{\iota}\alpha}$, $^{\iota}\alpha=\bigcup_{\xi<\iota}{}^{\xi}\alpha$ for $\iota$ a limit ordinal. 

As usual, we denote $^{\omega}\omega$ by $\varepsilon_{0}$. 
\end{defini}

Recall the following result from Koepke and Miller \cite{KM}:

\begin{defini}[Cf., e.g., \cite{Sa}, p. 48]
Let $x\subseteq\omega$. The hyperjump of $x$ is the set of all $i\in\omega$ such that the $i$-th Turing program computes a well-ordering in the oracle $x$.  
For $\iota<\varepsilon_{0}$, denote by HJ$^{\iota}(x)$ the $\iota$-th hyperjump of $x$; HJ$^{\iota}$ denotes the $\iota$-th hyperjump of $0$.
\end{defini}

\begin{thm}{\label{itrm hyperjump}}[See \cite{KM}, Theorem $1$]
There is an ITRM-program $P_{\text{hj}}$ such that, for each $x\subseteq\omega$, $P_{\text{hj}}^{x}$ computes HJ$(x)$. 
\end{thm}

%Also recall the following well-known result:
%
%\begin{lemma}
%Let $\alpha$ be a countable admissible ordinal, and let $c\subseteq\omega$ be a code for $\alpha$. Then $\omega_{1}^{
%\end{lemma}

\begin{corollary}
\begin{enumerate}
\item For any $n\in\omega$, the function $F_{n}:x\mapsto\text{HJ}_{\omega^{n}}(x)$, defined on $\mathfrak{P}(\omega)$, is $\omega^{\omega^{n}}$-ITRM-computable. 
\item For $n\in\omega$, we have $\beta(\omega^{\omega^{n}})\geq\omega_{\omega^{n+1}}^{\text{CK}}$. In particular, we have $\beta(\omega^{\omega})\geq\omega_{\omega^{2}}^{\text{CK}}$.
%\item For $n\in\omega$, $n>1$, we have $\beta(^{n}\omega)\geq\omega_{^{n-1}\omega}^{\text{CK}}$.
\item We have $\beta(\varepsilon_{0})\geq\omega_{\varepsilon_{0}\cdot\omega}^{\text{CK}}$. 
%\item For $n\in\omega$, we have $\beta(^{n}\omega)\geq\omega_{\omega^{n}}^{\text{CK}}$.
%\item $\beta(\varepsilon_{0})\geq\omega_{\omega^{\omega}}^{\text{CK}}$. 
\end{enumerate}
\end{corollary}
\begin{proof}
\begin{enumerate}
%\item By Lemma \ref{iteration lemma}, there is an $\omega^{\omega}$-ITRM-program $P_{\text{ihj}}$ which computes HJ$^{n}$ on input $n$. By running $P_{\text{ihj}}(i,j)$ on input $p_{\omega}(i,j)$, we obtain an $\omega^{\omega}$-ITRM-program $P_{\omega-\text{hj}}$ that computes the $\omega$-th hyperjump of $x$ in the oracle $x$. Using Lemma \ref{non-uniform finite iteration lemma}, we obtain that HJ$^{\omega\cdot k}$ is $\omega^{\omega}$-ITRM-computable for any $k\in\omega$. It follows that, for any $k\in\omega$, there is an $\omega^{\omega}$-ITRM-computable code for $\omega_{\omega\cdot k}^{\text{CK}}$. Consequently, the supremum $\beta(\alpha)$ of the ordinals with $\omega^{\omega}$-ITRM-computable codes is at least $\omega_{\omega^{2}}^{\text{CK}}$. %be more precise here, use spector-gandy

\item We prove this by induction. For $n=0$, this is Theorem \ref{itrm hyperjump}. Now suppose that $x\mapsto\text{HJ}_{\omega^{n}}(x)$ is $\omega^{\omega^{n}}$-ITRM-computable, say by the program $P_{n}$. By Lemma \ref{iteration lemma}, there is an $(\omega^{\omega^{n}})^{\omega}$-ITRM-program $Q$ that computes $F_{n}^{i}(x)$ on input $i\in\omega$; note 
that $(\omega^{\omega^{n}})^{\omega}=\omega^{\omega^{n+1}}$. 
 By running $Q(i,j)$ on input $p_{\omega}(i,j)$, we obtain an $\omega^{\omega^{n+1}}$-ITRM-program $Q^{\prime}$ that computes $F_{n}^{\omega}(x)$ in the oracle $x$. But $F_{n}^{\omega}$ is just $F_{n+1}$. 
%Using Lemma \ref{non-uniform finite iteration lemma}, we obtain that F$^{\omega^{n}\cdot k}$ is $\omega^{\omega^{n+1}}$-ITRM-computable for any $k\in\omega$. 

\item From (1), we have that HJ$_{\omega^{n}}$ is $\omega^{\omega^{n}}$-ITRM-computable; using Lemma \ref{non-uniform finite iteration lemma}, we obtain that HJ$_{\omega^{n}\cdot k}$ is $\omega^{\omega^{n}}$-ITRM-computable for every $k\in\omega$. Therefore, a code for $\omega_{\omega^{n}\cdot k}^{\text{CK}}$ is $\omega^{\omega^{n}}$-ITRM-computable for every $k\in\omega$. %be more precise here. use spector-gandy. 
Consequently, the supremum $\beta(\omega^{\omega^{n}})$ of the ordinals with $\omega^{\omega^{n}}$-ITRM-computable codes is at least $\omega_{\omega^{n+1}}^{\text{CK}}$.

%\item By (3), there is an $\varepsilon_{0}$-ITRM-computable code for $\omega_{^{k}\omega}^{\text{CK}}$ for every $k\in\omega$. Thus, the supremum $\beta(\varepsilon_{0})$ of ordinals with $\varepsilon_{0}$-ITRM-computable codes is at least their supremum, i.e., $\omega_{^{\omega}\omega}^{\text{CK}}=\omega_{\varepsilon_{0}}^{\text{CK}}$. 

\item By Theorem \ref{itrm hyperjump} and Corollary \ref{full iteration lemma}, HJ$_{\varepsilon_{0}\cdot k}$ is $\varepsilon_{0}$-ITRM-computable for any $k\in\omega$. 
Thus, $\beta(\varepsilon_{0})$ is larger than $\omega_{\varepsilon_{0}\cdot k}^{\text{CK}}$ for any $k\in\omega$, and thus $\beta(\varepsilon_{0})\geq\omega_{\varepsilon_{0}\cdot\omega}^{\text{CK}}$. 
\end{enumerate}
\end{proof}

The same approach works in a much more general situation: 

\begin{defini}
Let us say that $\alpha$ is ITRM-countable if and only if there is an $\alpha$-ITRM-computable bijection $f:\omega\rightarrow\alpha$. 
More generally, let us say that $\alpha$ is ITRM-effectively $\beta$-codable if and only if there is an $\alpha$-ITRM-computable bijection $f:\beta\rightarrow\alpha$.
\end{defini}

\begin{remark}
In particular, since subsets definable over $L_{\alpha}$ can always be computed on an $\alpha$-ITRM, $\alpha$ is ITRM-countable whenever $\alpha$ is an index (i.e., an ordinal $\alpha$ such that $(L_{\alpha+1}\setminus L_{\alpha})\cap\mathfrak{P}(\omega)\neq\emptyset$). 
Note that ITRM-countability implies that there is an $\alpha$-ITRM-computable real number that codes $\alpha$. 
\end{remark}

\begin{corollary}{\label{itrm countable}}
Let $\alpha>\omega$ be exponentially closed and ITRM-countable. Then $\beta(\alpha)\geq\alpha^{+\alpha\cdot\omega}$. 
\end{corollary}
\begin{proof}
%The proof is similar to that of Corollary $17$, now using transfinite induction.
 Let $x\subseteq\omega$ be an $\alpha$-ITRM-computable code for $\alpha$. By applying Corollary $8$ to $x$ and the ($\omega$-)ITRM-program that computes hyperjumps from Theorem \ref{itrm hyperjump}, we see that HJ$^{\alpha\cdot k}(x)$ is $\alpha$-ITRM-computable for every $k\in\omega$. 
But then, we have $\beta(\alpha)\geq\alpha^{+\alpha\cdot k}$ for every $k\in\omega$, i.e., $\beta(\alpha)\geq\alpha^{+\alpha\cdot\omega}$. 
\end{proof}

%\begin{remark}
%Note that the iteration technique just described never leads to a register overflow, so that the lower bounds just obtained in fact hold true already for the weak %(``unresetting'') $\alpha$-ITRMs as well that were mentioned in the introduction. In the case $\alpha=\omega$, it is known that $\alpha$-ITRMs are far %stronger than their unresetting cousins. We do not know whether the same is true for, e.g., $\alpha=\varepsilon_0$.
%\end{remark}

Besides raw size, one can also obtain some structural information on $\beta(\alpha)$ from these considerations. It was shown in Proposition $40$ of \cite{alpha itrms} that $\beta(\alpha)$ is never admissible and that, if $\alpha$ is an index, then $\beta(\alpha)$ is a limit of admissible ordinals. This can now be considerably strengthened. 

\begin{defini}
Let $\alpha$ be an ordinal.
We say that $\alpha$ is a level $0$ limit of admissible ordinals if and only if $\alpha$ is admissible. 
For $\iota\in\text{On}$, $\alpha$ is a level $\iota+1$ limit of admissible ordinals if and only if $\alpha$ is a limit of level $\iota$ limits of admissible ordinals. 
For $\delta\in\text{On}$ a limit ordinal, $\alpha$ is a level $\delta$ limit of admissible ordinals if and only if $\alpha$ is a level $\iota$ admissible ordinal for all $\iota<\delta$. 

We write Lev$(\alpha,\xi)$ to express that $\alpha$ is a level $\xi$-limit of admissible ordinals. 

Moreover, we write $[\alpha]^{\xi}_{\text{Adm}}$ for the smallest level $\xi$ limit of admissible ordinals that is strictly larger than $\alpha$ (thus, $[\alpha]^{0}_{\text{Adm}}=\alpha^{+}$) and $[\alpha]^{\xi,\iota}_{\text{Adm}}$ for the $\iota$-th smallest such limit. 
\end{defini}

\begin{lemma}{\label{computable limit iterations}}
For each $\iota\in\text{On}$, there is an $\omega^{\omega^{\iota}}$-ITRM-program $P_{\iota-\text{limit}}$ which, given a real number coding an ordinal $\alpha$, computes a real number coding $[\alpha]^{\iota}_{\text{Adm}}$. 
\end{lemma}
\begin{proof}
We prove the claim by induction on $\iota$. The case $\iota=0$ is just the fact that ($\omega$-)ITRMs can compute hyperjumps (and thus admissible successors). 
The limit case is trivial, since $\omega^{\omega^{\delta}}$-ITRMs can simulate $\omega^{\omega^{\iota}}$-ITRMs for all $\iota<\delta$ (uniformly on input $\iota$). We are thus left with the successor case. So suppose that the $\omega^{\omega^{\iota}}$-ITRM-program $P_{\iota-\text{limit}}$ is given, which computes a function $F$ as in the lemma. By Lemma \ref{iteration lemma}, we can compute the $\omega$-th iterate of $F$ on an $(\omega^{\omega^{\iota}})^{\omega}$-ITRM, i.e., on an $\omega^{\omega^{\iota+1}}$-ITRM. Given a real number $c$ coding an ordinal $\gamma$, this yields a real number that encodes 
$[\gamma]^{\iota,i}_{\text{Adm}}$ for all $i\in\omega$. From $c$, one recursively (in the classical sense, and uniformly in $c$) obtains a real number coding the ordinal sum $\rho:=\Sigma_{i\in\omega}[\gamma]^{\iota,i}_{\text{Adm}}$, which is equal to $[\gamma]^{\iota+1}_{\text{Adm}}$. 
\end{proof}

\begin{corollary}{\label{beta alpha limit}}
Let $\alpha>\omega$ be exponentially closed and ITRM-countable. %Then $\beta(\alpha)$ is a limit of limits of admissible ordinals. In fact, $\beta(\alpha)$ satisfies all $\iota$-iterations of ``limit of'' for all $\iota<\alpha$. [DO IT!!! GENAUER!!!]
Then $\beta(\alpha)$ is a level $\alpha$ limit of admissible ordinals. 
\end{corollary}
\begin{proof}
It follows from the ITRM-countability of $\alpha$ that in fact every ordinal smaller than $\beta(\alpha)$ has an $\alpha$-ITRM-computable real code: For, if $f:\omega\rightarrow\alpha$ is an $\alpha$-ITRM-computable bijection, then so is $f^{-1}:\alpha\rightarrow\omega$ and so, if $x\subseteq\alpha$ is any $\alpha$-ITRM-computable set coding an ordinal $\rho$, then $\{f^{-1}(\iota):\iota\in x\}$ is an ITRM-computable real number which also codes $\rho$. 

Since $\alpha$ is exponentially closed, $\alpha$ is a limit ordinal. Let $\iota<\alpha$. 
Hence, if $\gamma<\beta(\alpha)$, there is an $\alpha$-ITRM-computable real number $c$ that codes $\gamma$. By Lemma \ref{computable limit iterations}, $\hat{\gamma}:=[\gamma]^{\iota}_{\text{Adm}}$ is $\omega^{\omega^{\iota}}$-computable from the input $c$; since $\omega^{\omega^{\iota}}<\alpha$ by exponential closure, $\hat{\gamma}$ has an $\alpha$-ITRM-computable code, so $\hat{\alpha}<\beta(\alpha)$ and, by definition, $\hat{\alpha}$ is a level $\iota$ limit of admissible ordinals. Since $\gamma$ was arbitrary, $\beta(\alpha)$ must be a limit of such ordinals. 

\end{proof}

Analyzing the proof of Corollary \ref{itrm countable} -- and the auxiliary results leading there -- one notes that the iteration technique just described never leads to a register overflow, so that the lower bounds just obtained in fact hold true already for the $\alpha$-wITRMs as well:

\begin{corollary}{\label{wITRM countable}}
Let $\alpha>\omega$ be exponentially closed and wITRM-countable. Then $\beta^{w}(\alpha)\geq\alpha^{+\alpha\cdot\omega}$. 
\end{corollary}

 In the case $\alpha=\omega$, it is known that $\alpha$-ITRMs are far stronger than $\alpha$-wITRMs: Namely, the wITRM-computable subsets of $\omega$ are exactly those in $L_{\omega_{1}^{\text{CK}}}$, while the ITRM-computable ones are those in $L_{\omega_{\omega}^{\text{CK}}}$.  As we just noted, the techniques for obtaining lower bounds just described are insensitive to the distinction between resetting and unresetting machines. This leads to the following question:

\begin{question}
Are there any exponentially closed\footnote{Note that the examples given in Proposition \ref{weak equals strong} above are far from being exponentially closed.} values of $\alpha$ such that $\beta(\alpha)=\beta^{w}(\alpha)$, i.e. such that the computational strength of $\alpha$-ITRMs is the same as that of $\alpha$-wITRMs?
\end{question}

\subsection{Uncountable $\alpha$}

The lower bounds obtained from the iteration lemma above can only work when $\alpha$ is countable. In this section, we indicate how Abramson's and Sacks' ``lifting'' of results of Gostanian \cite{Go} on Gandy ordinals to the uncountable in \cite{AS} can be exploited to yield information on $\alpha$-ITRM-computability for certain uncountable values of $\alpha$. 
For the sake of brevity, simplicity and surveyability, we restrict ourselves to the case $\alpha=\aleph_{\omega}^{+}$ treated in \cite{AS}; further generalizations are deferred to later work. (The argument would equally well work for $(\aleph_{\omega}^{L})^{+}$.) 

In \cite{AS}, the authors prove that $\aleph_{\omega}^{+}$ is Gandy, i.e., that the supremum of the $\aleph_{\omega}^{+}$-recursive ordinals is $(\aleph_{\omega}^{+})^{+}$. Clearly, $\alpha$-recursive sets are also $\alpha$-ITRM-computable, and so this implies that $\beta(\aleph_{\omega}^{+})\geq(\aleph_{\omega}^{+})^{+}$; indeed, this much was observed in \cite{alpha itrms}. However, in order to use the strength of the iteration lemma, this is not enough: rather than being able to go from $\aleph_{\omega}^{+}$ to $(\aleph_{\omega}^{+})^{+}$, we would need a uniform way -- i.e., an $\alpha$-ITRM-program -- that allows us to go from some $x\subseteq\alpha$ that codes a well-ordering to $\omega_{1}^{\text{CK},x}$, i.e., the smallest ordinal $\beta>\alpha$ such that $L_{\beta}[x]$ is admissible. 

Such a program can indeed be obtained from the proof of Theorem $5$ of \cite{AS} by a relativization of the construction; we will offer a brief sketch of the general strategy and the necessary adaptations. 

We use the following generalization of Theorem $1$ of \cite{KM}: 

%\begin{defini}[\cite{C}, Definition 2.3.23]
%An ordinal $\alpha>\omega$ is ITRM-singular if and only if there is an $\alpha$-ITRM-computable cofinal function $f:\beta\rightarrow\alpha$ with $\beta<\alpha$.
%\end{defini}

\begin{lemma}{\label{infinite branch computation}}[See \cite{C}, Theorem 2.3.25]
If $\alpha$ is ITRM-singular, then there is an $\alpha$-ITRM-program $P_{\text{ifs}}$ (``ill-founded sequence'') such that, for any $x\subseteq\alpha$ that codes a tree $\mathcal{T}$ on $\alpha$, $P_{\text{ifs}}^{x}$ outputs $\emptyset$ when $\mathcal{T}$ is well-founded and otherwise outputs an infinite branch of $\mathcal{T}$.\footnote{More precisely, $P_{\text{ifs}}^{x}(i)$ will output the $i$-th element of an infinite branch of $\mathcal{T}$, for every $i\in\omega$.}
\end{lemma}

\begin{lemma}{\label{well-founded part computation}}
If $\alpha$ is ITRM-singular, then there is an $\alpha$-ITRM-program $P_{\text{wfp}}$ (``well-founded part'') such that, for any $x\subseteq\alpha$ that encodes a structure $(X,E)$, $P^{x}$ computes a subset of $\alpha$ that codes the well-founded part of $X$ with respect to $E$.
\end{lemma}
\begin{proof}
This follows from Lemma \ref{infinite branch computation} by cutting off the given structure $(X,E)$ below any given $x$ and applying the well-foundedness check to determine whether there is an infinite $E$-decreasing sequence that starts with $x$.
\end{proof}

%In the following, let $\alpha$ be $\aleph_{\omega}^{+}$. 

The general strategy in \cite{AS} is the following: They define an $\aleph_{\omega}^{+}$-recursive tree $\mathcal{T}$, guaranteed to have an infinite branch, whose infinite branches encode -- possibly ill-founded -- models of KP for which $\aleph_{\omega}^{+}$ belongs to the well-founded part. Since well-founded parts of admissible sets are known to be admissible, it follows that 
the height of the well-founded part of such a model must be of height at least $(\aleph_{\omega}^{+})^{+}$, from which one obtains the Gandyness of $\aleph_{\omega}^{+}$. 

It is not hard to modify their construction to obtain, for a given $x\subseteq\aleph_{\omega}^{+}$, a tree $\mathcal{T}_{x}$ that is uniformly $\aleph_{\omega}^{+}$-ITRM-computable in the oracle $x$, has at least one infinite branch and whose infinite branches encode models of KP whose well-founded part includes $\aleph_{\omega}^{+}$ and $x$. All that is required is to add, in the proof of Theorem $5$ of \cite{AS}, a new variable $\chi$ to the language $\mathcal{L}^{*}$ and the statements $\{d_{\gamma}\in\chi:\gamma\in x\}\cup\{d_{\gamma}\notin\chi:\gamma\notin\chi\}$ to the theory $\mathcal{T}^{*}$ and modify condition (viii) to demand that $(V,G)\in L_{\aleph_{\omega}^{+}}[x]$. The proof that the tree arising in this way has an infinite branch and that one obtains a model with the required properties from each infinite branch then works as in \cite{AS}. Now, by Lemma \ref{infinite branch computation}, we can uniformly compute a code $b\subseteq\aleph_{\omega}^{+}$ for such a branch on an $\aleph_{\omega}^{+}$-ITRM in $\mathcal{T}^{*}$. From $b$, one can then easily obtain a code $m\subseteq\aleph_{\omega}^{+}$ that encodes a model of KP with $\aleph_{\omega}^{+}$ and $x$ in its well-founded part. We can then use Lemma \ref{well-founded part computation} to compute a code $w\subseteq\aleph_{\omega}^{+}$ for the well-founded part of $m$. Using bounded truth predicate evaluation (see, e.g., \cite{C}, Theorem 2.3.28) in $m$, this yields a code for the set of ordinals in $m$, which will be a code of an ordinal $\geq\omega_{1}^{\text{CK},x}$. 

%Since this works for any $x\subseteq\aleph_{\omega}^{+}$, it is now possible to proceed as above to obtain $\beta((\aleph_{\omega}^{+})^+)\geq(\aleph_{\omega}^{+})^{+(\aleph_{\omega}^{+}\omega)}$. 

Since this works for any $x\subseteq\aleph_{\omega}^{+}$, it is now possible to proceed as above to obtain the following:

\begin{thm}
We have $\beta((\aleph_{\omega}^{+})^+)\geq(\aleph_{\omega}^{+})^{+(\aleph_{\omega}^{+}\cdot\omega)}$. 
\end{thm}

%ins apal-paper integriert
%\begin{lemma}
%Say that $\beta>\alpha$ is ``totally $\alpha$-safe'' if and only if $\beta$ is a limit ordinal and, for every $\alpha$-ITRM-program $P$ and any $P$-configuration $c$, $P^{c}$ does not have a proper limit or an overflow at time $\beta$. 
%Then $\beta=\beta(\alpha)$. 
%\end{lemma}
%\begin{proof}
%As in the base case of the inductive proof in the APAL-paper: To see that direct limits of configuration appear early enough, count upwards in some register whenever each register in turn has the right value and use the fact that $\omega$ cannot appear as a proper limit at time $\beta$. 
%\end{proof}

%das klappt nicht unbedingt für schwache ITRMs, weil die berechnung des wfp evtl. den overflow braucht.

\section{Open Questions}

While the above refutes a natural conjecture on the computational strength of $\alpha$-ITRMs by providing some lower bounds, the value of $\beta(\alpha)$ is still unknown for all values of $\alpha$ unless $\alpha=\omega$ or $L_{\alpha}\models$ZF$^{-}$. Some special cases that might be good starting points would be 
to determine $\beta(\omega^{\omega})$, $\beta(\varepsilon_{0})$, $\beta(\aleph_{\omega})$ or $\beta(\omega_{1}^{\text{CK}})$. 

%As a specific question, we ask whether the following holds:

%\begin{conjecture}
%For every exponentially closed ordinal $\alpha$, we either have $L_{\alpha}\models\text{ZF}^{-}$ or $\beta(\alpha)=\alpha^{+\alpha\omega}$.
%\end{conjecture}

A crucial feature of $\omega$-ITRMs established by Koepke and Miller in \cite{KM}, the generalization of which may well shed light on the computational power of $\alpha$-ITRMs, is the solvability of the bounded halting problem. Although we are able to prove that, for each ordinal $\alpha$, there is either a universal $\alpha$-ITRM-program or the bounded halting problem for $\alpha$-ITRMs is solvable, we are in a quite unsatisfying situation: We do not know which alternative holds for a single exponentially closed ordinal $\alpha$ except when $\alpha=\omega$ or when $L_{\alpha}\models$ZF$^{-}$ 
which alternative holds. A crucial step in further work on the computational strength of $\alpha$-ITRMs might be to generalize the work on the cases $\alpha=\omega$ and $L_{\alpha}\models$ZF$^{-}$ by seeing whether the computational strength of $\alpha$-ITRMs can be characterized by iterating some operator 
that is $\beta$-ITRM-computable for some $\beta\leq\alpha$. We also currently do not know whether there are values of $\alpha$ for which the lower bounds obtained in this paper are optimal. We expect that proof-theoretical considerations on iterated admissibility and inductive operators such as J\"ager \cite{Jaeger} and \cite{BFPS} will become relevant in further investigations. 

For the time being, we thus restrict ourselves to the following rather modest questions: 

\begin{question}
Determine $\beta(\alpha)$ or $\beta^{w}(\alpha)$ for any value of $\alpha$ other than $\alpha=\omega$ or $\alpha$ a ZF$^{-}$-ordinal. 
\end{question}

\begin{question}
Characterize the $u$-weak ordinals, i.e., those for which $\beta^{w}(\alpha)=\alpha$ (and thus, COMP$_{\alpha-\text{wITRM}}=\Delta_{1}(L_{\alpha})$).
\end{question}

\section{Acknowledgements}

We thank the three anonymous referees of \cite{C2} for their valuable feedback, in particular for pointing out several subtle typos.

 \bibliographystyle{splncs04}

%\bibliographystyle{splncs04}
%\bibliography{REFERENCES}

\end{document}